\documentclass[12pt]{article}
\usepackage{amsmath,amssymb}
\usepackage{color,graphicx}
\newcommand{\bs}{\boldsymbol}
\newcommand{\mbf}{\mathbf}

\newcommand{\bb}{\mathbf b}

\newcommand{\be}{\mathbf e}
\newcommand{\bE}{\mathbf E}

\newcommand{\bv}{\mathbf v}

\newcommand{\bw}{\mathbf w}
\newcommand{\bu}{\mathbf u}
\newcommand{\bx}{\mathbf x}
\newcommand{\bg}{{\boldsymbol g}}

\newtheorem{theorem}{Theorem}
\newtheorem{lema}{Lemma}
\newcounter{remark}
\def\theremark {\arabic{remark}}
\newenvironment{remark}{\refstepcounter{remark}\par\noindent{\bf Remark\ \theremark}\ }{\par}
\newtheorem{Proof}{Proof}

\newenvironment{proof}{\begin{Proof}\rm}{\hfill $\Box$ \end{Proof}}

\title{Error analysis of projection methods for non inf-sup stable mixed finite elements. The transient Stokes problem.
}
\author{Javier de Frutos\thanks{Instituto de Investigaci\'on en Matem\'aticas (IMUVA),
Universidad de Valladolid, Spain.  Research supported by Spanish MINECO
under grants MTM2013-42538-P and MTM2016-78995-P (AEI/FEDER, UE) (frutos@mac.uva.es)} \and Bosco
Garc\'{\i}a-Archilla\thanks{Departamento de Matem\'atica Aplicada
II, Universidad de Sevilla, Sevilla, Spain. Research supported by
Spanish MINECO under grant MTM2015-65608-P (bosco@esi.us.es)}
  \and Julia Novo\thanks{Departamento de
Matem\'aticas, Universidad Aut\'onoma de Madrid, Instituto de
Ciencias Matem\'aticas CSIC-UAM-UC3M-UCM, Spain. Research supported
by Spanish MINECO
under grants MTM2013-42538-P and MTM2016-78995-P (AEI/FEDER, UE) (julia.novo@uam.es)}}
\date{}
\begin{document}
\maketitle
\begin{abstract}
A modified Chorin-Teman (Euler non-incremental) projection method and a modified Euler incremental projection method for non inf-sup stable mixed finite elements are analyzed. The analysis of the classical Euler non-incremental and Euler incremental methods are obtained as a particular case. We first prove that the modified Euler non-incremental scheme has an inherent stabilization that allows the use of non inf-sup stable mixed finite elements without any kind of extra added stabilization. We show that it is also true in the case of the classical Chorin-Temam method. For the second scheme, we study a stabilization that allows the use of equal-order pairs of finite elements. The relation of the methods with the so called pressure stabilized Petrov Galerkin method (PSPG) is established. The influence of the chosen initial approximations  in the computed approximations to the pressure is analyzed. Numerical tests confirm the theoretical results.
\end{abstract}

{\bf keywords}
Projection methods, PSPG stabilization, non inf-sup stable elements
\bigskip


\section{Introduction}
In this paper we analyze a  modified Chorin-Temman (Euler non-incremental) projection method for non inf-sup stable mixed finite elements. The analysis of the classical Euler non-incremental  method is obtained as a particular case. We prove that both the  modified and the standard Euler non-incremental schemes have an inherent stabilization that allows the use of non inf-sup stable mixed finite elements without any kind of extra added stabilization. Although this result is known (see for example \cite{Guermond_Quar_IJNMF}) to our knowledge there are no proved error bounds for the Chorin-Temam method with non inf-sup stable elements in the literature (see below for related results in~\cite{badia-codina}).
For the closely-related Euler incremental scheme we analyze a modified method for non inf-sup stable pairs of finite elements. In this case an added stabilization is required. The analysis of a stabilized Euler incremental scheme is also obtained as a consequence of the analysis of the modified method.
We establish the  relation of the methods with the so called pressure stabilized Petrov Galerkin method (PSPG).

It has been observed in the literature that the standard Euler non-incremental scheme provides computed pressures that behave unstably for $\Delta t$ small and fixed $h$ if non inf-sup stable elements are used, see~\cite{codina}. With our error analysis we clarify this question since in that case the inherent PSPG stabilization of the method disappears.

In the present paper, we analyze the influence of the initial approximations
to the velocity and pressure in the error bounds for the pressure.
In agreement with the results obtained for the PSPG method in \cite{John_Novo_PSPG} a stabilized Stokes approximation of the initial data is suggested as initial approximation. We show both analytically and numerically that with this initial approximation we can obtain accurate approximations for the pressure from the first time step.

Our analysis is valid for any pair
of  non inf-sup stable  mixed finite elements whenever the pressure space $Q_h$ satisfies the condition $Q_h\subset H^1(\Omega)$. However, we prove that the rate of convergence cannot be
better than quadratic (in terms of $h$) for the $L^2$ errors of the velocity and linear for  the $L^2$ errors of the pressure so that using finite elements other than linear elements in the approximations to the velocity and pressure offers no clear advantage. In terms of $\Delta t$ the rate of convergence we prove is one for the $L^2$ errors of the velocity. For the $L^2$ discrete in time and $H^1$ in space errors for the velocity and $L^2$ discrete in time and $L^2$ in  space errors for the pressure the rate of convergence in terms of $\Delta t$ is one for the modified Chorin-Temam method and is one half for the standard Chorin-Temam method, accordingly to the rate of convergence of the continuous in space Chorin-Temam method, see \cite{Guermond_overview} and the references therein. The analysis presented in this paper is not intended to obtain bounds with constants independent of the viscosity parameter. The possibility of obtaining viscosity independent error bounds will be the subject of further research.

Of course, the Chorin-Temam projection method  is well known and this is not the first paper where the analysis of this method is considered. The analysis of the semidiscretization in time  is carried out in \cite{shen1}, \cite{shen2}, \cite{ranacher}, \cite{Prohl}, \cite{Prohl2}. In \cite{codina} the stability of the Chorin-Temam
projection method is considered and, in case of non inf-sup stable mixed finite elements, some a priori bounds
for the approximations to the velocity and pressure are obtained
but no error bounds are proven for this method. In  \cite{badia-codina} the Chorin-Teman method is considered together with both non inf-sup stable and inf-sup stable mixed finite elements.
In case of using non inf-sup stable mixed finite elements a local projection type stabilization is required in \cite{badia-codina} to get the
error bounds of the method. In the present paper, however, we get optimal error bounds without any extra stabilization for non inf-sup stable mixed finite elements.

For the Euler incremental scheme the analysis of the semidiscretization in time  can be found in~\cite{Prohl}.
The Euler incremental scheme with a spatial discretization based on inf-sup stable mixed finite elements is analyzed in \cite{Guermond_Quar2}. To our knowledge there is no error analysis for this method in case of using non-inf-sup stable elements.  Some stability estimates can be found in \cite{codina} for the method with added stabilization terms more related to local projection stabilization than to the PSPG stabilization we consider in the present paper.
A stabilized version of  the incremental scheme is also proposed in \cite{minev} although no error bounds are proved.
Finally, for an overview on projection methods we refer the reader to \cite{Guermond_overview}.

Being the Chorin-Temam projection method an old one, it has seen the appearance of many alternative methods during the years, many of which possess better convergence properties. The purpose of this
paper is not to discuss its advantages of disadvantages with respect to newer methods, but just to analyze its inherent stabilization properties
which allow the use of  non inf-sup stable elements without extra stabilization, and its connection with (more modern) PSPG stabilization.

For simplicity in the exposition we  concentrate in this paper in the transient Stokes equations assuming enough regularity for the solution.
In~\cite{Bosco-JuliaII} we extend the analysis to the Navier-Stokes equations in the general case in which non-local compatibility conditions for the solution are not assumed.

The outline of the paper is as follows. We first introduce some notation. In the second section we consider the steady Stokes equations and introduce a stabilized Stokes approximation that will be used in the error analysis of the method. Next section is devoted to the analysis of the evolutionary Stokes equations assuming enough regularity for the solution. Both methods  Euler non-incremental and Euler-incremental schemes are considered. In the last section some numerical experiments are shown.

\section{Preliminaries and notation}
Throughout the paper, standard notation is used for Sobolev
spaces and corresponding norms. In particular, given a measurable set $\omega\subset\mathbb{R}^d$, $d=2,3$, its Lebesgue
measure is denoted by~$|\omega|$,  the inner product in $L^2(\omega)$ or $L^2(\omega)^d$ is
denoted by $(\cdot,\cdot)_\omega$ and the notation $(\cdot,\cdot)$ is used instead
of $(\cdot,\cdot)_\Omega$.  The semi norm in $W^{m,p}(\omega)$ will be denoted
by $|\cdot|_{m,p,\omega}$ and, following~\cite{Constantin-Foias}, we define the norm~$\left\|\cdot\right\|_{m,p,\omega}$ as
$$
\left\| f\right\|_{m,p,\omega}^p=\sum_{j=0}^m \left|\omega\right|^{\frac{p(j-m)}{d}} \left| f\right|_{j,p,\omega}^p,
$$
so that $\left\|f\right\|_{m,p,\omega}\left|\omega\right|^{\frac{m}{d}-\frac{1}{p}}$ is scale invariant. We will also use
the conventions $\|\cdot\|_{m,\omega}=\|\cdot\|_{m,2,\omega}$ and
$\|\cdot\|_{m}=\|\cdot\|_{m,2,\Omega}$. As it is usual we will use the special notation $H^s(\omega)$ to denote $W^{s,2}(\omega)$ and
we will denote by $H_0^1(\Omega)$ the subspace of functions of  $H^1(\Omega)$  satisfying homogeneous Dirichlet boundary conditions. Finally, $L^2_0(\Omega)$ will denote the subspace of function of $L^2(\omega)$ with zero mean.

Let us denote by $\mathcal T_h$ a triangulation of the domain $\Omega$, which, for simplicity, is assumed to have a Lipschitz polygonal boundary.
On $\mathcal T_h$, we consider the  finite element spaces $V_h \subset V=H_0^1(\Omega)^d$ and
$Q_h \subset L_0^2(\Omega) \cap H^1(\Omega)$ based on local polynomials of degree $k$ and
$l$ respectively. Equal degree polynomials for velocity and pressure are allowed. It will be assumed in the rest of the paper that the family of
meshes is regular.

We will denote by $J_h \bu \in V_{h}$ the elliptic projection of a function $\bu\in V$ defined by
$$ (\nabla (\bu-J_h \bu),\nabla\bv_h)=0,\quad\forall\bv_h\in V_h.$$
The following bound holds
for $m=0,1$ and $\bu \in H^{k'+1}(\Omega)^d$, $0\le k'\le k$,
\begin{equation}\label{cota_inter}
\|\bu-J_h\bu\|_m\le Ch^{k'+1-m}\|\bu\|_{k'+1}.
\end{equation}

Analogously, we will denote by $J_h z\in Q_h$ the elliptic projection of a function $z \in H^1(\Omega)$. For $m=0,1$, $z\in H^{l'+1}(\Omega)$, $0\le l'\le l$ it holds
\begin{eqnarray}
\|z-J_hz\|_m&\le& C h^{l'+1-m}\|z\|_{l'+1},\label{eq:cota_ellip}\\
\|J_h z\|_1&\le& C \|z\|_1\label{eq:estabi_eli}.
\end{eqnarray}

The following inverse
inequality holds for each $v_{h} \in V_{h}$, see e.g., \cite[Theorem 3.2.6]{Cia78},
\begin{equation}
\label{inv} \| \bv_{h} \|_{W^{m,p}(K)} \leq C_{\mathrm{inv}}
h_K^{n-m-d\left(\frac{1}{q}-\frac{1}{p}\right)}
\|\bv_{h}\|_{W^{n,q}(K)},
\end{equation}
where $0\leq n \leq m \leq 1$, $1\leq q \leq p \leq \infty$, and $h_K$
is the size (diameter) of the mesh cell $K \in \mathcal T_h$.

Let~$\lambda$ be the smallest eigenvalue of  $A=-\Delta$ subject to homogeneous Dirichlet  boundary conditions,
$\Delta$ being the Laplacian operator in~$\Omega$. Then it is well-known that there exists a scale-invariant positive constant~$c_{-1}$ such that
\begin{equation}
\label{eq:cota_menos1}
\left\| \bv\right\|_{-1}\le c_{-1}\lambda^{-1/2} \left\| \bv\right\|_0,\qquad
\bv\in L^{2}(\Omega)^d,
\end{equation}
and, also,
\begin{equation}
\label{eq:poincare}
\left\| \bv\right\|_{0}\le \lambda^{-1/2} \left\|\nabla \bv\right\|_0,\qquad
\bv\in H^1_0(\Omega)^d,
\end{equation}
this last inequality is also known as the Poincar\'e inequality.

\section{A stabilized Stokes projection}

Let us consider the Stokes problem
\begin{eqnarray}\label{eq:stokes}
-\nu\Delta {\mbf s}+\nabla z&=&\hat \bg,\quad {\rm in}\quad \Omega\nonumber\\
\nabla \cdot {\mbf s}&=& 0,\quad {\rm in}\quad \Omega\\
{\mbf s}&=&{\bs 0},\quad{\rm on}\quad \partial \Omega.\nonumber
\end{eqnarray}
We define the stabilized Stokes approximation to (\ref{eq:stokes}) as the mixed
finite element approximation $({\mbf s}_h,z_h)\in (V_h,Q_h)$ satisfying
\begin{eqnarray}\label{eq:pro_stokes}
\nu(\nabla {\mbf s}_h,\nabla {\bs\chi}_h)+(\nabla z_h,{\bs\chi}_h)&=&(\hat\bg,{\bs\chi}_h),\quad \forall {\bs\chi}_h\in V_h,\\
(\nabla \cdot {\mbf s}_h,\psi_h)&=&-\delta(\nabla z_h,\nabla \psi_h),\quad \forall \psi_h\in Q_h,\label{eq:pro_stokes2}
\end{eqnarray}
where $\delta$ is a constant parameter. Observe that from~\eqref{eq:stokes} and~\eqref{eq:pro_stokes} it follows that the errors
$s_h-s$ and~$z_h-z$ satisfy that
\begin{equation}
\label{eq:gal_orthog}
\nu(\nabla({\mbf s}_h-{\mbf s}),\nabla{\bs\chi}_h) + (\nabla (z_h-z),{\bs\chi}_h)=0,\qquad \forall{\bs\chi}_h\in V_h.
\end{equation}

The pair $(J_h {\mbf s},J_h z)$ satisfies the following equations for all ${\bs\chi}_h\in V_h$ and $\psi_h\in Q_h$
\begin{eqnarray}\label{eq:inter_elli}
\nu(\nabla J_h{\mbf s},\nabla {\bs\chi}_h)+(\nabla J_h z,{\bs\chi}_h)&=&(\hat\bg,{\bs\chi}_h)-(T_{1},\nabla \cdot {\bs\chi}_h), \\
(\nabla \cdot J_h{\mbf s},\psi_h)&=&-\delta(\nabla J_h z,\nabla \psi_h)+
\delta (T_2,\nabla \psi_h), \nonumber
\end{eqnarray}
where  $T_{1}$ and~$T_2$ are the truncation errors
\begin{eqnarray}\label{eq:T1_2}
T_{1}=J_h z-z,\quad T_2=
\frac{{\mbf s}-J_h{\mbf s}}{\delta}+\nabla J_hz.
\end{eqnarray}
Let us denote by
$$
\be_h= {\mbf s}_h- J_h {\mbf s},\quad r_h=z_h-J_h z,
$$
subtracting (\ref{eq:inter_elli}) from (\ref{eq:pro_stokes}) it is easy to reach
\begin{eqnarray}
\nu(\nabla \be_h,\nabla{\bs\chi}_h)+(\nabla r_h,{\bs\chi}_h)&=&(T_{1},\nabla \cdot {\bs\chi}_h),\quad \forall{\bs\chi}_h\in V_h\qquad \label{eq:er1si}\\
( \nabla \cdot \be_h,\psi_h)&=&-\delta(\nabla r_h,\nabla \psi_h)-\delta(T_2,\nabla \psi_h),\quad \forall \psi_h \in Q_h.\nonumber
\end{eqnarray}
Taking ${\bs\chi}_h=\be_h$ and~$\psi_h=r_h$  we obtain
\begin{align*}
\nu\|\nabla \be_h\|_0^2+\delta \|\nabla r_h\|_0^2&\le \nu^{-1}\| T_1\|_0^2+
\delta \| T_2\|_0^2.
\end{align*}
In view of the expressions of~$T_1$ and~$T_2$ in~(\ref{eq:T1_2}), the right hand side above can be bounded in terms of
 $\nu^{-1}\|J_hz -z\|_0^2 +
\delta^{-1}\| J_h{\mbf s}-{\mbf s}\|_0^2 +\delta\left\|\nabla J_hz\right\|_0^2,
$
so that denoting
\begin{equation}
\label{eq:los_eMes}
M({\mbf s},z):=\nu^{-1/2}\left\| J_hz-z\right\|_0+\delta^{-1/2} \| J_h{\mbf s}-{\mbf s}\|_0,
\end{equation}
and recalling  (\ref{eq:estabi_eli}) we have
\begin{equation}\label{eq_cota_final_steady}
\nu\|\nabla \be_h\|_0^2+\delta \|\nabla r_h\|_0^2\le
C\bigl( M({\mbf s},z) +\delta^{1/2}\left\|\nabla z\right\|_0\bigr)^2.
\end{equation}
Using the triangle inequality we obtain
\begin{equation}\label{eq_cota_final_steady_def}
\nu^{1/2}\|\nabla ({\mbf s}-{\mbf s}_h)\|_0+\delta^{1/2}\|\nabla (z-z_h)\|_0
\le  C \bigl(M({\mbf s},z)+\delta^{1/2}\left\|\nabla z\right\|_0\bigr).
\end{equation}

In the sequel we set
\begin{equation}
\label{eq:rho}
\rho=\frac{h}{(\nu\delta)^{1/2}},
\end{equation}
so that applying~(\ref{cota_inter}) and (\ref{eq:cota_ellip}) we have the estimate
\begin{equation}\label{eq_cota_final_steady2_orig}
M({\mbf s},z)\le C\Bigl(\rho \nu^{1/2}h^{k'}\|{\mbf s} \|_{k'+1}
+\frac{h^{l'+1}}{\nu^{1/2}}\|z\|_{l'+1} \Bigr).
\end{equation}

To bound $\|r_h\|_0$ we will use the following lemma \cite[Lemma 3]{Burman}, \cite[Lemma 2.1]{John_Novo_PSPG}.
\begin{lema}\label{lema_presion}
For $\psi_h\in Q_h $ it holds
\begin{equation}\label{eq:L2_pre}
\|\psi_h\|_0\le C h\|\nabla \psi_h\|_0+C\sup_{{\bs\chi}_h\in V_h}\frac{(\psi_h,\nabla \cdot {\bs\chi}_h)}{\|{\bs\chi}_h\|_1}.
\end{equation}
\end{lema}
Applying (\ref{eq:L2_pre}), (\ref{eq:er1si}) and (\ref{eq_cota_final_steady}) we get
\begin{eqnarray*}
\|r_h\|_0&\le& C h \delta^{-1/2} \delta^{1/2}\|\nabla r_h\|_0+\sup_{{\bs\chi}_h\in V_h}\frac{(r_h,\nabla \cdot {\bs\chi}_h)}{\|{\bs\chi}_h\|_1}\nonumber\\
&\le& C h \delta^{-1/2} \delta^{1/2}\|\nabla r_h\|_0+\nu\|\nabla e_h\|_0+\|T_{1}\|_0\nonumber\\
&\le& C \bigl(( h \delta^{-1/2}+\nu^{1/2})M({\mbf s},z)
+(h+(\nu\delta)^{1/2})\left\|\nabla z\right\|\Bigr).
\end{eqnarray*}
Applying the triangle inequality 
we have
\begin{equation}\label{eq:errorpre_l2}
\|z-z_h\|_0
\le C \nu^{1/2}( 1+\rho)\bigl(M({\mbf s},z)+\delta^{1/2}\left\|\nabla z\right\|_0\bigr).
\end{equation}
To conclude this section we will get a bound for the $L^2$ norm of the error  by means of a well-known duality argument.

\begin{lema}\label{le:duality} There exist a constant~$C>0$ such that
for any $\bv\in H^1_0(\Omega)^d$ with $\hbox{\rm div}(\bv)=0$, $q\in~L^2_0(\Omega)$,  $\bv_h\in V_h$ and~$q_h\in Q_h$ satisfying
\begin{eqnarray}
\label{eq:gal_orthogv}
\nu(\nabla(\bv_h-\bv),\nabla{\bs\chi}_h) + (\nabla (q_h-q),{\bs\chi}_h)=0,\qquad \forall{\bs\chi}_h\in V_h,\\
\label{eq:gal_orthogv2}
(\nabla\cdot(\bv_h-\bv),\psi_h)+\delta(\nabla q_h,\nabla\psi_h)=0,\qquad
\forall \psi_h\in Q_h,
\end{eqnarray}
the following bound holds:
\begin{equation}
\left\| \bv_h-\bv\right\|_0 \le C\left(h\left(\left\|\nabla(\bv-\bv_h)\right\|_0+\nu^{-1}\left\|q-q_h\right\|_0\right)+\delta\left\|\nabla q_h\right\|_0\right).
\label{eq:cota2}
\end{equation}

\end{lema}
\begin{proof} To prove~(\ref{eq:cota2}), for~${\bs \phi}=\bv-\bv_h$, let $(\bE,Q)$ be the solution of
\begin{equation}\label{eq:dual}
\begin{array}{rclcl}
-\nu\Delta \bE+\nabla Q&=&{\bs \phi},&\quad&{\rm in}\quad \Omega,\\
\nabla \cdot \bE&=&0,&\quad&{\rm in}\quad \Omega,\\
\bE&=&0,&\quad&{\rm on}\quad \partial\Omega.\end{array}
\end{equation}
Since we are assuming $\Omega$ is smooth enough the solution of (\ref{eq:dual}) satisfies
\begin{equation}\label{eq:auxdual2}
\nu\|\bE\|_2+\|Q\|_1\le C\|{\bs \phi}\|_0= C\|\bv-\bv_h\|_0.
\end{equation}
Then, we have
\begin{equation}\label{eq:auxdual1}
\|\bv-\bv_h  \|_0^2=({\bs\phi}, \bv-\bv_h )=\nu(\nabla (\bv-\bv_h),\nabla \bE)-(\nabla \cdot(\bv-\bv_h),Q).
\end{equation}
For the first term on the right-hand side of (\ref{eq:auxdual1}) adding and subtracting $J_h \bE$ and  using (\ref{eq:gal_orthogv})  we get
\begin{align*}
\nu(\nabla (\bv-\bv_h),\nabla \bE)&=\nu(\nabla (\bv-\bv_h),\nabla(\bE-J_h \bE))+\nu(\nabla (\bv-\bv_h),\nabla J_h \bE)
\nonumber\\&=\nu(\nabla (\bv-\bv_h),\nabla(\bE-J_h \bE))+(q-q_h,\nabla\cdot J_h\bE).
\nonumber\\&=\nu(\nabla (\bv-\bv_h),\nabla(\bE-J_h \bE))+(q-q_h,\nabla\cdot (\bE-J_h\bE)).
\end{align*}
Then, applying (\ref{cota_inter}) and (\ref{eq:auxdual2}) we obtain
\begin{align}\label{align1}
\nu(\nabla (\bv-\bv_h),\nabla \bE)&\le \Bigl(\nu\|\nabla (\bv-\bv_h)\|_0
+
C\|q-q_h\|_0\Bigr)\|\nabla(\bE-J_h \bE)\|_0\nonumber\\
&\le
\left(
\|\nabla (\bv-\bv_h)\|_0+C\nu^{-1}\|q-q_h\|_0\right)h\nu\|\bE\|_2
\\
&\le C h\left(
\|\nabla (\bv-\bv_h)\|_0+C\nu^{-1}\|q-q_h\|_0\right)\|\bv-\bv_h\|_0.
\nonumber
\end{align}
For the second term on the right-hand side of (\ref{eq:auxdual1})
we add and subtract $J_h Q$ and apply (\ref{eq:gal_orthogv2})
\begin{align*}
(\nabla  \cdot(\bv-{\bv_h}),Q)&=(\nabla  \cdot(\bv-{\bv_h}),Q-J_h Q)+
(\nabla  \cdot(\bv-{\bv_h}),J_h Q)\nonumber\\
&=(\nabla  \cdot(\bv-{\bv_h}),Q-J_h Q)+\delta(\nabla q_h,\nabla J_h Q).\nonumber
\end{align*}
Applying now (\ref{eq:cota_ellip}) and (\ref{eq:estabi_eli}) together with (\ref{eq:auxdual2}) we get
\begin{align}\label{align2}
(\nabla  \cdot(\bv-{\bv_h}),Q)&\le \|\nabla (\bv-{\bv_h})\|_0\|Q-J_h Q\|_0+\delta \|\nabla q_h\|_0\|\nabla J_h Q\|_0\nonumber\\
&\le C\left(h\|\nabla (\bv-{\bv_h})\|_0
+\delta\|\nabla q_h\|_0\right)\| Q\|_1\\
&\le C\left(h\|\nabla (\bv-{\bv_h})\|_0
+\delta\|\nabla q_h\|_0\right)\| \bv-\bv_h\|_0\nonumber
\end{align}
Inserting (\ref{align1}) and (\ref{align2}) into (\ref{eq:auxdual1}) we reach~(\ref{eq:cota2})
\end{proof}

We now apply~(\ref{eq:cota2}) with~$\bv={\bs s}$, $q=z$, $\bv_h={\bs s}_h$ and~$q_h=z_h$ to get
$$
\left\| {\mbf s}-{\mbf s}_h\right\|_0 \le C\bigl(h\left(\left\|\nabla({\mbf s}-{\mbf s}_h)\right\|_0+\nu^{-1}\left\|z-z_h\right\|_0\right)+\delta\left\|\nabla z_h\right\|_0\bigr).
$$
Applying~(\ref{eq_cota_final_steady_def}) and~(\ref{eq:errorpre_l2}) together with definition (\ref{eq:rho})
we get
\begin{align*}
\left\| {\mbf s}-{\mbf s}_h\right\|_0 &\le C
\left({h}{\nu^{-1/2}}(2+\rho)\left(M({\mbf s},z)+\delta^{1/2}\left\|\nabla z\right\|_0\right)+\delta\left\|\nabla z_h\right\|_0\right)\nonumber\\
&\le C
\left(\rho(2+\rho)\delta^{1/2}\left(M({\mbf s},z)+\delta^{1/2}\left\|\nabla z\right\|_0\right)+\delta\left\|\nabla z_h\right\|_0\right).
\end{align*}
By writing~$\delta\left\|\nabla z_h\right\|\le \delta\left\|\nabla (z_h-z)\right\|
+\delta\left\|\nabla z\right\|$ and applying~(\ref{eq_cota_final_steady_def})
we have
\begin{equation}\label{eq_cota_final_steady_def_l2_b}
\|{\mbf s}-{\mbf s}_h\|_0 \le  C
(1+\rho )^2\delta^{1/2}\left(M({\mbf s},z)+\delta^{1/2}\left\|\nabla z\right\|_0\right),
\end{equation}
and applying~(\ref{eq_cota_final_steady2_orig}),
\begin{align}
\label{eq:una_mas}
\|{\mbf s}-{\mbf s}_h\|_0 \le &C(1+\rho)^2\Bigl({h^{k'+1}}\|{\mbf s}\|_{k'+1}
+\frac{\delta^{1/2}}{\nu^{1/2}} h^{l'+1 }\|z\|_{l'+1}+\delta\|\nabla z\|_0\Bigr),
\end{align}
for~$0\le k'\le k$ and~$0\le l'\le l$.

We notice that in the last bound there are positive  powers of
the parameter~$\rho$. This implies that  in order to have optimal error bounds in the velocity $\rho$ must be bounded above. Hence,
in the sequel, we will assume
\begin{equation}\label{eq:cond_rho}
\rho\le \rho_1,
\end{equation}
for a positive constant  $\rho_1$ which implies
\begin{equation}\label{eq:cond_delta}
\frac{1}{\nu\rho_1^2}h^2\le \delta.
\end{equation}
Assuming~(\ref{eq:cond_rho})
we obtain the following simplified error bounds for $0\le k'\le k$ and $0\le l'\le l$.
\begin{align}\label{eq:error_steady_simplified}
\nu^{1/2}\|\nabla({\mbf s}-{\mbf s}_h)\|_0+\delta^{1/2}\|\nabla(z-z_h)\|_0\le \frac{C}{\nu^{1/2}} &\hat M_1({\mbf s},z),\nonumber\\
\|z-z_h\|_0\le C &\hat M_1({\mbf s},z),\nonumber\\
\|{\mbf s}-{\mbf s}_h\|_0\le C &\hat M_2({\mbf s},z),
\end{align}
where the constants $C$ in the bounds above depend on the value~$\rho_1$ in~(\ref{eq:cond_rho}), and
\begin{eqnarray*}
\hat M_1({\mbf s},z)&=&\nu h^{k'}\|{\mbf s}\|_{k'+1}+h^{l'+1}\|z\|_{l'+1}+(\nu\delta)^{1/2}\|\nabla z\|_0,\\
\hat M_2({\mbf s},z)&=&h^{k'+1}\|{\mbf s}\|_{k'+1}+\nu^{-1}h^{2l'+2}\|z\|_{l'+1}+\delta(\|\nabla z\|_0+\|z\|_{l'+1}),
\end{eqnarray*}
where $\hat M_2$ is otained from~(\ref{eq:una_mas}) by writing $\frac{\delta^{1/2}}{\nu^{1/2}} h^{l'+1 } \le
\frac{\delta}{2} + \frac{h^{2(l'+1)}}{2\nu}$.
We observe that independently of the degree of the piecewise polynomials, in view of condition (\ref{eq:cond_delta}),   we do not achieve more than second order in the $L^2$ norm
of the error of the velocity and first order in the $L^2$ norm of the error of the pressure due to the terms $\delta\|\nabla z\|_0$ and $\delta^{1/2}\|\nabla z\|_0$ respectively.
Using piecewise linear polynomials both in the approximations to the velocity and the pressure (i.e. with $k=l=1$) and assuming $({\mbf s},z)\in H^2(\Omega)^d\times H^1(\Omega)$ (i.e. taking $l'=0$)
we get
\begin{eqnarray}\label{eq:error_steady_simplified_linear}
\nu^{1/2}\|\nabla({\mbf s}-{\mbf s}_h)\|_0+\delta^{1/2}\|\nabla(z-z_h)\|_0&\le& C \frac{h}{\nu^{1/2}}(\nu \|{\mbf s}\|_{2}+\|z\|_1)+C\delta^{1/2}\|z\|_1,\nonumber\\
\|z-z_h\|_0&\le&C h(\nu\|{\mbf s}\|_{2}+\|z\|_1)+C(\nu\delta)^{1/2}\|z\|_1,\\
 \|{\mbf s}-{\mbf s}_h\|_0&\le&C\frac{h^{2}}{\nu}(\nu\|{\mbf s}\|_{2}+\|z\|_{1})+C\delta \|z\|_1,\nonumber
\end{eqnarray}
the constants~$C$ depending on the value~$\rho_1$ in~(\ref{eq:cond_rho}). Here and in the rest of the paper we use $C$ to denote a generic non-dimensional constant.

\section{Evolutionary Stokes equations}
In the rest of the paper we consider the evolutionary Stokes equations
\begin{eqnarray}\label{eq:evo_stokes}
\bv_t-\nu\Delta {\bv}+\nabla q&=&{\mbf g},\quad {\rm in}\quad \Omega\nonumber\\
\nabla \cdot {\bv}&=& 0,\quad {\rm in}\quad \Omega\\
{\bv}&=&{\bs 0},\quad{\rm on}\quad \partial \Omega,\nonumber\\
\bv(0,\bx)&=&\bv_0(\bx),\quad {\rm in}\quad \Omega.\nonumber
\end{eqnarray}
We will introduce a modified Euler non incremental scheme in the first part of this section and we will end the section considering a modified Euler incremental scheme. The error analysis of the second scheme is obtained as a consequence of the error analysis of the first method.
\subsection{Euler non-incremental scheme}
We will denote by $(\bv_h^n,\tilde \bv_h^n,q_h^n)$, $n=1,2,\ldots,$\  $\tilde \bv_h^n\in V_h$, $q_h^n\in Q_h$ and $ \bv_h^n\in V_h+\nabla Q_h$
the approximations to the velocity and pressure at time $t_n=n\Delta t$, $\Delta t=T/N$, $N>0$ obtained with the following modified Euler non-incremental scheme
\begin{eqnarray}\label{eq:eu_non}
&&\left(\frac{\tilde \bv_h^{n+1}-\bv_h^n}{\Delta t},{\bs\chi}_h\right)+\nu(\nabla \tilde \bv_h^{n+1},\nabla {\bs\chi}_h)=({\mbf g}^{n+1},{\bs\chi}_h),\quad \forall {\bs\chi}_h\in V_h\nonumber\\
&&(\nabla \cdot \tilde \bv_h^{n+1},\psi_h)=-\delta(\nabla q_h^{n+1},\nabla \psi_h),\quad \forall \psi_h\in Q_h,\\
&&\bv^{n+1}_h=\tilde \bv_h^{n+1}-\delta\nabla q_h^{n+1}.\nonumber
\end{eqnarray}
Let us observe that for $\delta=\Delta t$,  (\ref{eq:eu_non}) is the classical  Chorin-Temam (Euler non-incremental) scheme \cite{Chorin}, \cite{Temam}. In case $\delta=\Delta t$ we can remove $\bv_h^n$ from (\ref{eq:eu_non}) inserting the expression of $\bv_h^n$ from the last equation in (\ref{eq:eu_non}) into the first equation in (\ref{eq:eu_non}) to get
\begin{eqnarray}\label{eq:eu_non_tilde}
&&\left(\frac{\tilde \bv_h^{n+1}-\tilde \bv_h^n}{\Delta t},{\bs\chi}_h\right)+\nu(\nabla \tilde \bv_h^{n+1},\nabla {\bs\chi}_h)+(\nabla q_h^n,{\bs \chi}_h)=({\mbf g}^{n+1},{\bs\chi}_h),\quad \forall {\bs\chi}_h\in V_h,\qquad\\
&&(\nabla \cdot \tilde \bv_h^{n+1},\psi_h)=-\delta(\nabla q_h^{n+1},\nabla \psi_h),\quad \forall \psi_h\in Q_h.\label{eq:eu_non_tilde2}
\end{eqnarray}
The method we study is exactly (\ref{eq:eu_non_tilde})-(\ref{eq:eu_non_tilde2}) with  $\delta$ a parameter not necessarily equal to $\Delta t$.
More precisely, we suggest to take $\delta$ as defined in (\ref{eq:cond_delta}). Let us observe that in the formulation (\ref{eq:eu_non_tilde})-(\ref{eq:eu_non_tilde2}) we only look for approximations $\tilde \bv_h^n\in V_h$ and $q_h^n\in Q_h$ to the velocity and pressure respectively.
The discrete divergence free approximation $\bv_h^n$ to the velocity is  not part of the scheme. As a consequence of the error analysis of this section we will get the error bounds for the classical Euler non-incremental scheme assuming in that case $\delta=\Delta t$.
\begin{remark}
Let us observe that condition (\ref{eq:eu_non_tilde2}) is analogous to the condition imposed for the pressure stabilized Petrov-Galerkin (PSPG) method to stabilize non inf-sup stable mixed-finite elements, see \cite{John_Novo_PSPG}. The difference is that in the PSPG method
instead of (\ref{eq:eu_non_tilde2}) one has the full residual
\begin{eqnarray}
(\nabla \cdot \tilde \bv_h^{n+1},\psi_h)&=&\delta\sum_{K\in \mathcal{T}_h}\left(({\mbf g}^{n+1},\nabla \psi_h)_K-\left(\frac{\tilde \bv_h^{n+1}-\tilde \bv_h^n}{\Delta t},\nabla \psi_h\right)_K\right.\nonumber\\
&&\quad -(\nu \Delta \tilde \bv_h^{n+1},\psi_h)_K-(\nabla q_h^{n+1},\nabla \psi_h)_K\left.\right)
\end{eqnarray}
so that the PSPG method is consistent, while in (\ref{eq:eu_non_tilde2})  we only keep the last term on the right-hand side above which is the one giving stability  for the approximate pressure. However, due to the lack of consistency no better that $O(h^2)$ error bounds can be obtained for the method (\ref{eq:eu_non_tilde})-(\ref{eq:eu_non_tilde2}). The analogy between the PSPG method and the modified Euler non-incremental scheme applies also to the value of the stabilization parameter $\delta$ which is in general for the PSPG method $\delta \approx h^2$, see \cite{John_Novo_PSPG}. Let us observe that we assume a lower bound for $\delta$  of size $h^2$ in (\ref{eq:cond_delta}) for the method (\ref{eq:eu_non_tilde})-(\ref{eq:eu_non_tilde2}). In view of (\ref{eq:error_steady_simplified_linear}) assuming also an analogous upper bound, i.e. $\delta \approx h^2$, gives an error $O(h)$ for the first two bounds in (\ref{eq:error_steady_simplified_linear}) and $O(h^2)$ for the last one so that assumption $\delta \approx h^2$ equilibrates all terms in (\ref{eq:error_steady_simplified_linear}).
\end{remark}

Let us denote by ${\mbf g}^n={\mbf g}(t_n)$ and $\bv_t^n=\bv_t(t_n)$. Let us consider  $({\mbf s}_h^n,z_h^n)$  the stabilized Stokes approximation to the steady Stokes problem (\ref{eq:stokes}) with
right-hand side $\hat{\mbf  g}^n={\mbf g}^n-\bv_t^n$. Let us observe that  $(\bv^n,p^n)=(\bv(t_n),p(t_n))$, i.e., the solution of the evolutionary Stokes problem (\ref{eq:evo_stokes}) at time $t=t_n$ is also the exact solution of this steady problem.
  More precisely,  $({\mbf s}_h^n,z_h^n)\in V_h\times Q_h$ satisfies
\begin{eqnarray}\label{eq:pro_stokes_evo}
\nu(\nabla {\mbf s_h^n},{\bs\chi}_h)+(\nabla z_h^n,{\bs\chi}_h)&=&(\hat{\mbf g}^n,{\bs\chi}_h),\quad {\bs\chi}_h\in V_h,\\
(\nabla \cdot {\mbf s_h^n},\psi_h)&=&-\delta(\nabla z_h^n,\nabla \psi_h),\quad \forall \psi_h\in Q_h\nonumber.
\end{eqnarray}
In the sequel we will denote by
\begin{equation}\label{eq:defer}
\tilde \be_h^n=\tilde \bv_h^n-{\mbf s}_h^n,\quad  r_h^n=q_h^n-z_h^n.
\end{equation}
From  (\ref{eq:eu_non_tilde})-(\ref{eq:eu_non_tilde2}) and (\ref{eq:pro_stokes_evo}) one obtains the following error equation for all ${\bs\chi}_h\in V_h$, $\psi_h\in Q_h$
\begin{align}\label{eq:er1_orig}
\left(\frac{\tilde \be_h^{n+1}-\tilde \be_h^n}{\Delta t},{\bs\chi}_h\right)+\nu (\nabla \tilde \be_h^{n+1},\nabla {\bs\chi}_h)+(\nabla r_h^n,{\bs\chi}_h)&=({\boldsymbol\tau}_h^n,{\bs\chi}_h)-(\nabla (z^n_h-z_h^{n+1}),{\bs\chi}_h),\qquad\qquad\\
(\nabla \cdot \tilde \be_h^{n+1},\psi_h)+\delta(\nabla r_h^{n+1},\nabla \psi_h)&=0.
\label{eq:er3_orig}
\end{align}
where
\begin{equation}
\label{eq:tau_h}
{\boldsymbol\tau}_h^n=\bv_t^{n+1}-\frac{{\mbf s}_h^{n+1}-{\mbf s}_h^n}{\Delta t}=(\bv_t^{n+1}-({\mbf s_h})_t^{n+1})+\left(({\mbf s_h})_t^{n+1}-\frac{{\mbf s}_h^{n+1}-{\mbf s}_h^n}{\Delta t}\right).
\end{equation}

To estimate the errors $\tilde\be_h^n$ and~$r_h^n$ we will use the following stability result.


\begin{lema}\label{lema_stab_evol} Let $(\bw_h^n)_{n=0}^\infty$
and~$(\bb_h^n)_{n=0}^\infty$ sequences in~$V_h$
and~$(y_h^n)_{n=0}^\infty$ and~$(d_h^n)_{n=0}^\infty$ sequences in
$Q_h$ satisfying for all ${\bs\chi}_h\in V_h$ and $\psi_h\in Q_h$
\begin{align}\label{eq:er1}
\left(\frac{\bw_h^{n+1}-\bw_h^n}{\Delta t},{\bs\chi}_h\right)+\nu (\nabla \bw_h^{n+1},\nabla {\bs\chi}_h)+(\nabla y_h^n,{\bs\chi}_h)=&(\bb_h^n+\nabla d_h^n,{\bs\chi}_h),\\
\label{eq:er3}
(\nabla \cdot \bw_h^{n+1},\psi_h)+\delta(\nabla y_h^{n+1},\nabla \psi_h)=&0.{}\quad
\end{align}
Assume condition
\begin{equation}\label{eq:cond_delta2}
 \Delta t \le \delta
\end{equation}
holds.
Then, for $0\le n_0\le n-1$ there exits a non-dimensional constant $c_0$ such that the following bounds hold
\begin{align}\label{eq:stab_evol_1}
\|\bw_h^{n}\|_0^2+\sum_{j=n_0}^{n-1}\|\bw_h^{j+1}-\bw_h^j&\|_0^2
+\Delta t\sum_{j=n_0}^{n-1}\bigl(\nu\|\nabla\bw_h^{j+1}\|_0^2+{\delta}\|\nabla y_h^{j+1}\|_0^2\bigr)
\nonumber\\
&\le c_0\left(\|\bw_h^{n_0}\|_0^2+ \Delta t\sum_{j=n_0}^{n-1}\left(\nu^{-1}\|\bb_h^j\|_{-1}^2+\delta\|\nabla d_h^j\|_0^2\right)\right).
\end{align}
\begin{align}\label{eq:stab_evol_1t}
t_n\|\bw_h^{n}\|_0^2&+\sum_{j=n_0}^{n-1}{t_{j+1}}\|\bw_h^{j+1}-\bw_h^j\|_0^2
+\Delta t\sum_{j=n_0}^{n-1}t_{j+1}\bigl(\nu\|\nabla\bw_h^{j+1}\|_0^2+{\delta}\|\nabla y_h^{j+1}\|_0^2\bigr)
\nonumber\\
&\le c_0\left(t_{n_0}\|\bw_h^{n_0}\|_0^2+\Delta t\sum_{j=n_0}^{n}\|\bw_h^{j}\|_0^2
+ \Delta t\sum_{j=n_0}^{n-1}
t_{j+1}\left(t_{j+1}\|\bb_h^j\|_{0}^2+\delta\|\nabla d_h^j\|_0^2\right)\right).
\end{align}
\begin{align}
\label{eq:stab_evol_2}
\sum_{j=n_0}^{n-1}{\Delta t}&\biggl\|\frac{\bw_h^{j+1}-\bw_h^{j}}{\Delta t}\biggr\|_0^2+\nu\|\nabla \bw_h^{n}\|_0^2+ \delta\|\nabla y_h^{n}\|_0^2+\nu\sum_{j=n_0}^{n-1}\|\nabla(\bw_h^{j+1}-\bw_h^j)\|_0^2\nonumber\\
&{}\quad\le
c_0\left(\nu\|\nabla \bw_h^{n_0}\|_0^2+\delta\|\nabla y_h^{n_0}\|_0^2
+ \Delta t\sum_{j=n_0}^{n-1}\left(\|\bb_h^j\|_0^2+\|\nabla d_h^j\|_0^2\right)\right).\quad
\end{align}
\end{lema}
\begin{proof}
Taking ${\bs\chi}_h=\Delta t\bw_h^{n+1}$ in~(\ref{eq:er1}) and~$\psi_h=\Delta ty_h^n$  in~(\ref{eq:er3}) we get
\begin{align}\label{eq:er2}
\frac{1}{2}\left(\|\bw_h^{n+1}\|_0^2-\|\bw_h^n\|_0^2+\|\bw_h^{n+1}-\bw_h^n\|_0^2\right)
+\nu\Delta t \|\nabla &\bw_h^{n+1}\|_0^2+\Delta t(\nabla y_h^n,\bw_h^{n+1})\qquad\qquad\nonumber \\
&\le \Delta t(\bb_h^n+\nabla d_h^n,\bw_h^{n+1}),\\
\Delta t(\nabla \cdot \bw_h^{n+1},y_h^n)+\delta \Delta t(\nabla y_h^{n+1},\nabla y_h^n)&=0.
\label{eq:er4}
\end{align}
Summing both equations and noticing
that, after integration by
parts,~$\Delta t(\nabla y_h^n,\bw_h^{n+1})$ in~(\ref{eq:er2}) cancels out with the
term~$\Delta t(\nabla \cdot \bw_h^{n+1},y_h^n)$
in~(\ref{eq:er4}), we
have
$$
\displaylines{
\frac{1}{2}\Bigl(\|\bw_h^{n+1}\|_0^2-\|\bw_h^n\|_0^2+\|\bw_h^{n+1}-\bw_h^n\|_0^2\Bigr)
+\nu\Delta t\|\nabla \bw_h^{n+1}\|_0^2+\delta \Delta t (\nabla y_h^{n+1},\nabla y_h^n)\hfill\cr
\hfill\le \Delta t(\bb_h^n+\nabla d_h^n,\bw_h^{n+1}).
}
$$
Multiplying by~$2$ and
adding and subtracting $\|\nabla y_h^{n+1}\|_0^2$ we get
\begin{align}
\label{eq:er5_orig}
\|\bw_h^{n+1}\|_0^2&-\|\bw_h^n\|_0^2+\|\bw_h^{n+1}-\bw_h^n\|_0^2
+2\nu\Delta t\|\nabla \bw_h^{n+1}\|_0^2+2\delta\Delta t\|\nabla y_h^{n+1}\|_0^2
\nonumber\\
&\le 2\Delta t(\bb_h^n,\bw_h^{n+1})+2\Delta t(\nabla d_h^n,\bw_h^{n+1})
+2\delta \Delta t(\nabla y_h^{n+1},\nabla (y_h^{n+1}-y_h^n)).
\end{align}
From (\ref{eq:er3}) it is also easy to obtain
$$
\delta(\nabla (y_h^{n+1}-y_h^n),\nabla y_h^{n+1})=-(\nabla \cdot( \bw_h^{n+1}-\bw_h^n),y_h^{n+1}),
$$
so that
$$
2\delta\Delta t(\nabla y_h^{n+1},\nabla (y_h^{n+1}-y_h^n))\le
\frac{2}{3}\|\bw_h^{n+1}-\bw_h^n\|_0^2+\frac{3}{2}(\Delta t)^2\|\nabla y_h^{n+1}\|_0^2.
$$
Thus, from~(\ref{eq:er5_orig}) and (\ref{eq:cond_delta2}) we have
\begin{align}\label{eq:er5}
\|\bw_h^{n+1}\|_0^2-\|\bw_h^n\|_0^2+\frac{1}{3}\|\bw_h^{n+1}-\bw_h^n\|_0^2
&+2\nu\Delta t\|\nabla \bw_h^{n+1}\|_0^2+\frac{1}{2}\delta\Delta t\|\nabla y_h^{n+1}\|_0^2
\nonumber\\
&\le 2\Delta t(\bb_h^n,\bw_h^{n+1})+2\Delta t(\nabla d_h^n,\bw_h^{n+1}).
\end{align}
We now bound the two terms on the right-hand side above.
For the first one we write
$$
2\Delta t(\bb_h^n,\bw_h^{n+1}) \le  \frac{\Delta t}{\nu}\|\bb_h^n\|_{-1}^2+\nu\Delta t\|\nabla \bw_h^{n+1}\|_0^2.
$$
For the second one we have $2\Delta t(\nabla d_h^n,\bw_h^{n+1})=-2\Delta t(d_h^n,\nabla\cdot\bw_h^{n+1})$, so that
using (\ref{eq:er3}) with $\psi_h=2\Delta t d_h^n$ we may write
\begin{eqnarray}
\label{eq:cota_segundo_er5}
2\Delta t(\nabla d_h^n,\bw_h^{n+1})=2\delta\Delta t(\nabla y_h^{n+1},\nabla d_h^n)
\le \frac{\delta\Delta t}{4}\|\nabla y_h^{n+1}\|_0^2 +4\delta\Delta t\|\nabla d_h^n\|_0^2.\quad\quad
\end{eqnarray}
Using the two inequalities above in (\ref{eq:er5}) we obtain
\begin{align*}
\|\bw_h^{n+1}\|_0^2-\|\bw_h^n\|_0^2+&\frac{1}{3}\|\bw_h^{n+1}-\bw_h^n\|_0^2
+2\nu\Delta t\|\nabla \bw_h^{n+1}\|_0^2+\frac{1}{2}\delta\Delta t\|\nabla y_h^{n+1}\|_0^2
\nonumber\\
\le& \frac{\Delta t}{\nu}\|\bb_h^n\|_{-1}^2+\nu\Delta t\|\nabla \bw_h^{n+1}\|_0^2
+4\delta\Delta t\|\nabla d_h^n\|_0^2
+\frac{\delta\Delta t}{4}\|\nabla y_h^{n+1}\|_0^2.\qquad
\end{align*}
Arranging terms we get
\begin{align}\label{eq:er6}
\|\bw_h^{n+1}\|_0^2-\|\bw_h^n\|_0^2+\frac{1}{3}\|\bw_h^{n+1}-\bw_h^n\|_0^2
+\nu\Delta t\|\nabla& \bw_h^{n+1}\|_0^2+\frac{1}{4}\delta\Delta t\|\nabla y_h^{n+1}\|_0^2\nonumber\\
&\le \frac{\Delta t}{\nu}\|\bb_h^n\|_{-1}^2+4\delta\Delta t\|\nabla d_h^n\|_0^2,
\end{align}
so that~(\ref{eq:stab_evol_1}) follows easily.

To prove~(\ref{eq:stab_evol_1t}), multiply~(\ref{eq:er5}) by~$t_{n+1}$
and write
$$
t_{n+1}\|\bw_h^n\|_0^2=
t_{n}\|\bw_h^n\|_0^2 + \Delta t\|\bw_h^n\|_0^2.
$$
Use~(\ref{eq:cota_segundo_er5}) to bound the
term~$2 t_{n+1}\Delta t(\nabla d_h^n,\bw_h^{n+1})$,
and for~$2 t_{n+1}\Delta t(\bb_h^n,\bw_h^{n+1})$
use the following bound
$$
2t_{n+1}\Delta t(\bb_h^n,\bw_h^{n+1}) \le  t_{n+1}^{2}\Delta t\|\bb_h^n\|_{0}^2+\Delta t\|
\bw_h^{n+1}\|_0^2,
$$
so that
\begin{align*}
t_{n+1}\|\bw_h^{n+1}\|_0^2&-t_n\|\bw_h^n\|_0^2\\
&+t_{n+1}\Bigl(\frac{1}{3}\|\bw_h^{n+1}-\bw_h^n\|_0^2
+2\nu\Delta t\|\nabla \bw_h^{n+1}\|_0^2+\frac{1}{4}\delta\Delta t\|\nabla y_h^{n+1}\|_0^2\Bigr)\\
&\le t_{n+1}\Delta t \left(t_{n+1}\|\bb_h^n\|_{0}^2+4\delta\|\nabla d_h^n\|_0^2\right)
+\Delta t\left(\|\bw_h^{n}\|_0^2+\|\bw_h^{n+1}\|_0^2\right),
\end{align*}
and~(\ref{eq:stab_evol_1t})
follows by summing consecutive values of~$n$.

To prove~(\ref{eq:stab_evol_2}) we take ${\bs\chi}_h=\bw_h^{n+1}-\bw_h^n$ in~(\ref{eq:er1}).
Then
\begin{equation}\label{eq:er10}
\begin{split}
\Delta t\biggl\|\frac{\bw_h^{n+1}-\bw_h^{n}}{\Delta t}\biggr\|_0^2&+\frac{\nu}{2}\left(\|\nabla \bw_h^{n+1}\|_0^2-\|\nabla \bw_h^n\|_0^2+\|\nabla(\bw_h^{n+1}-\bw_h^n)\|_0^2\right)\\
&+(\nabla y_h^n,\bw_h^{n+1}-\bw_h^{n})\\
&= (\bb_h^n,\bw_h^{n+1}-\bw_h^{n})+( \nabla d_h^n,\bw_h^{n+1}-\bw_h^{n}).
\end{split}
\end{equation}
For the last term on the left-hand side of (\ref{eq:er10}) applying (\ref{eq:er3}) we obtain
\begin{equation}\label{eq:er12}
\begin{split}
(\nabla y_h^n,\bw_h^{n+1}-\bw_h^{n})&=-(y_h^n,\nabla \cdot(\bw_h^{n+1}-\bw_h^n))
=\delta(\nabla y_h^n,\nabla (y_h^{n+1}-y_h^n))\\
&=\frac{\delta}{2}\left(\|\nabla y_h^{n+1}\|_0^2-\|\nabla y_h^n\|_0^2-\|\nabla(y_h^{n+1}-y_h^n)\|_0^2\right).
\end{split}
\end{equation}
We will bound the last term on the right-hand side above applying (\ref{eq:er3}) again:
\begin{align*}
{\delta}\|\nabla (y_h^{n+1}-y_h^n)\|_0^2&=
-(\nabla\cdot(\bw_h^{n+1}-\bw_h^n),y_h^{n+1}-y_h^n)\nonumber\\
&=(\bw_h^{n+1}-\bw_h^n,\nabla (y_h^{n+1}-y_h^n))
\nonumber\\&\le \frac{\delta}{2}\|\nabla (y_h^{n+1}-y_h^n)\|_0^2+\frac{1}{2\delta} \|\bw_h^{n+1}-\bw_h^n\|_0^2,
\end{align*}
so that
$$
\frac{\delta}{2}\|\nabla (y_h^{n+1}-y_h^n)\|_0^2\le \frac{1}{2\delta} \|\bw_h^{n+1}-\bw_h^n\|_0^2
=\frac{\delta}{2} \biggl\|\frac{\bw_h^{n+1}-\bw_h^n}{\delta}\biggr\|_0^2.
$$
Inserting the above inequality into (\ref{eq:er12}) we reach
\begin{eqnarray}\label{eq:er12_2}
(\nabla y_h^n,\bw_h^{n+1}-\bw_h^{n})\ge\frac{\delta}{2}\left(\|\nabla y_h^{n+1}\|_0^2-\|\nabla y_h^n\|_0^2\right)
-\frac{\delta}{2} \biggl\|\frac{\bw_h^{n+1}-\bw_h^n}{\delta}\biggr\|_0^2,
\end{eqnarray}
so that from~(\ref{eq:er10}) it follows that
\begin{equation*}
\begin{split}
2\Delta t\biggl\|\frac{\bw_h^{n+1}-\bw_h^{n}}{\Delta t}\biggr\|_0^2&-\delta\biggl\|\frac{\bw_h^{n+1}-\bw_h^{n}}{\delta}\biggr\|_0^2\\
&+\nu\left(\|\nabla \bw_h^{n+1}\|_0^2-\|\nabla \bw_h^n\|_0^2+\|\nabla(\bw_h^{n+1}-\bw_h^n)\|_0^2\right)\\
&+ \delta\left(\|\nabla y_h^{n+1}\|_0^2-\|\nabla y_h^n\|_0^2\right)\\
&\le 2(\bb_h^n,\bw_h^{n+1}-\bw_h^{n})+2( \nabla d_h^n,\bw_h^{n+1}-\bw_h^{n}).
\end{split}
\end{equation*}
Using from now on that restriction (\ref{eq:cond_delta2}) holds
we get
\begin{equation}\label{eq:er10bis}
\begin{split}
\Delta t\biggl\|\frac{\bw_h^{n+1}-\bw_h^{n}}{\Delta t}\biggr\|_0^2&+\nu\left(\|\nabla \bw_h^{n+1}\|_0^2-\|\nabla \bw_h^n\|_0^2+\|\nabla(\bw_h^{n+1}-\bw_h^n)\|_0^2\right)\\
&+ \delta\left(\|\nabla y_h^{n+1}\|_0^2-\|\nabla y_h^n\|_0^2\right)\\
&\le 2(\bb_h^n,\bw_h^{n+1}-\bw_h^{n})+2( \nabla d_h^n,\bw_h^{n+1}-\bw_h^{n}).
\end{split}
\end{equation}
To conclude we bound the two terms on the right-hand side above. For the first one we write
$$
2(\bb_h^n,\bw_h^{n+1}-\bw_h^{n})\le \frac{\Delta t}{4}\biggl\|\frac{\bw_h^{n+1}-\bw_h^{n}}{\Delta t}\biggr\|_0^2+4\Delta t\|\bb_h^n\|_0^2,
$$
and for the second one,
$$
2( \nabla d_h^n,\bw_h^{n+1}-\bw_h^{n})
\le \frac{\Delta t}{4}\biggl\|\frac{\bw_h^{n+1}-\bw_h^{n}}{\Delta t}\biggr\|_0^2+4\Delta t\|\nabla d_h^n\|_0^2.\nonumber
$$
Using these two bounds
in~(\ref{eq:er10bis}) we reach
\begin{equation*}
\begin{split}
\frac{\Delta t}{2}\biggl\|\frac{\bw_h^{n+1}-\bw_h^{n}}{\Delta t}\biggr\|_0^2&+\nu\left(\|\nabla \bw_h^{n+1}\|_0^2-\|\nabla \bw_h^n\|_0^2+\|\nabla(\bw_h^{n+1}-\bw_h^n)\|_0^2\right)\\
&+\delta\left(\|\nabla y_h^{n+1}\|_0^2-\|\nabla y_h^n\|_0^2\right)\\
&\le 4\Delta t\|\bb_h^n\|_0^2+4\Delta t\|\nabla d_h^n\|_0^2,
\end{split}
\end{equation*}
from where~(\ref{eq:stab_evol_2}) follows easily.
\end{proof}

\begin{remark} At the price of a more elaborate proof, it is possible
to replace condition~(\ref{eq:cond_delta2}) by~$\Delta t \le 2\delta$.
\end{remark}

We now prove a bound for the error in the velocity and pressure in the approximation defined by (\ref{eq:eu_non_tilde})-(\ref{eq:eu_non_tilde2}). We assume the solution
$(\bv,q)$ of (\ref{eq:evo_stokes}) is smooth enough so that all the norms appearing below on the right-hand side of the bounds in Theorem~\ref{Th1} are bounded.
\begin{theorem}\label{Th1} Let $(\bv,q)$ be the solution of (\ref{eq:evo_stokes}) and let $(\tilde \bv_h^n,q_h^n)$, $n\ge 1$,  be the solution of (\ref{eq:eu_non_tilde})-(\ref{eq:eu_non_tilde2}). Assume $\delta$ satisfies condition (\ref{eq:cond_delta}) and $\Delta t$ satisfies condition (\ref{eq:cond_delta2}).
Then, the following bounds hold
\begin{equation}\label{cota_th1_l2}
\begin{split}
\|\tilde \bv_h^n-\bv(t_n)\|_0^2\le C\|\tilde\be_h^0\|_0^2&+C\frac{h^4}{\nu^2}\left(\nu^2\|\bv(t_n)\|_2^2+\|q(t_n)\|_1^2\right)+C\delta^2\|q(t_n)\|_1^2\\
&+C_1^n t_n {\Delta t}^2+C_2^n t_n h^4+C_3^n(\nu\lambda)^{-1} t_n\delta^2,
\end{split}
\end{equation}
\begin{equation}\label{cota_th1_H1}
\begin{split}
\Delta t\sum_{j=1}^{n}\bigl(\nu\|\nabla (\tilde \bv_h^j&-\bv(t_j))\|_0^2+\delta \|\nabla (q_h^j-q(t_j))\|_0^2\bigr)
\\
&\le C\|\tilde\be_h^0\|_0^2
+C t_n\frac{h^2}{\nu}\left(\nu^2\max_{t_1\le t\le t_n}\left(\|\bv(t)\|_2^2+\|q(t)\|_1^2\right)\right)\\
&+Ct_n\delta\max_{t_1\le t\le t_n}\|q(t)\|_1^2+C_1^n t_n {\Delta t}^2+C_2^n t_n h^4+C_3^n(\nu\lambda)^{-1} t_n \delta^2,
\end{split}
\end{equation}
where $C_1^n$ and $C_2^n$ are defined as
\begin{eqnarray}
C_1^n&=&C\left(\frac{c_{-1}^2}{\nu\lambda}\max_{0\le t\le t_n}\|({\mbf s}_h)_{tt}(t)\|_0^2+\delta\max_{0\le t\le t_n}\|\nabla (z_h)_t(t)\|_0^2\right),\label{laC1}\\
C_2^n&=&\frac{C}{\nu^3\lambda}\left(\nu^2\max_{t_1\le t\le t_n}\left(\|\bv_t(t)\|_2^2+\|q_t(t)\|_1^2\right)\right),\label{laC2}\\
C_3^n&=&{C}\left(\max_{t_1\le t\le t_n}\|q_t(t)\|_1^2\right)\label{laC3}.
\end{eqnarray}
\end{theorem}
\begin{proof} We apply Lemma~\ref{lema_stab_evol} to relation~(\ref{eq:er1_orig})-(\ref{eq:er3_orig}), that is,
taking~$\bw_h^n=\tilde\be_h^n$, $y_h^n=r_h^n$, $\bb_h^n=P_{V_h}{\boldsymbol\tau}_h^n$
and~$d_h^n=z_h^{n+1} - z_h^n$, where $P_{V_h}$ is the $L^2$ orthogonal projection onto $V_h$. As a consequence
of~(\ref{eq:stab_evol_1}) we have
\begin{equation}\label{eq:err_evol_a}
\begin{split}
\|\tilde\be_h^{n}\|_0^2&+\sum_{j=n_0}^{n-1}\|\tilde\be_h^{j+1}-\tilde\be_h^j\|_0^2
+\Delta t\sum_{j=n_0}^{n-1}\bigl(\nu\|\nabla\tilde\be_h^{j+1}\|_0^2+{\delta}\|\nabla r_h^{j+1}\|_0^2\bigr)\\
&\le c_0\left(\|\tilde\be_h^{n_0}\|_0^2+ \Delta t\sum_{j=n_0}^{n-1}\left(\nu^{-1}\|P_{V_h}{\boldsymbol\tau}_h^j\|_{-1}^2+\delta\|\nabla (z_h^{j+1}-z_h^j)\|_0^2\right)\right).
\end{split}
\end{equation}
We now estimate the last two terms  on the right-hand side above. For the second one we have
\begin{equation}\label{eq:preer13}
\left\| \nabla (z_h^{j+1}-z_h^{j})\right\|_0^2=\bigg\|\int_{t_j}^{t_{j+1}}\nabla (z_h)_t~dt\biggr\|_0^2\le \Delta t\int_{t_j}^{t_{j+1}}\|\nabla (z_h)_t\|_0^2~dt,
\end{equation}
where in the last inequality we have applied H\"older's inequality. Thus we
can write
\begin{equation}\label{eq:er13}
\sum_{j=0}^{n-1}\delta\left\| \nabla (z_h^{j+1}-z_h^{j})\right\|_0^2
\le \delta \Delta t\int_{0}^{t_n}\|\nabla (z_h)_t\|_0^2~dt.
\end{equation}
To estimate the truncation error we first consider the second term in the expression of~${\boldsymbol\tau}_h^j$ in~(\ref{eq:tau_h}).
Applying H\"older's inequality we may write
\begin{eqnarray}\label{eq:trunca}
\left\|({\mbf s_h})_t^{j+1}-\frac{{\mbf s}_h^{j+1}-{\mbf s}_h^j}{\Delta t}\right\|_0^2
=\biggl\|\frac{1}{\Delta t}\int_{t_j}^{t_{j+1}} (t_{j}-s) ({\mbf s}_h)_{tt}\,dt\biggr\|_0^2
\le
\Delta t\int_{t_j}^{t_{j+1}} \left\|({\mbf s}_h)_{tt}\right\|_0^2\,dt,\qquad\quad
\end{eqnarray}
so that, recalling~(\ref{eq:cota_menos1}) and applying (\ref{eq:trunca}) we obtain
\begin{equation}\label{eq:er8-1}
\begin{split}
\|P_{V_h}{\boldsymbol\tau}_h^j\|_{-1}^2&\le \frac{c_{-1}^2}{\lambda}\|P_{V_h}{\boldsymbol\tau}_h^j\|_0^2 \le  \frac{c_{-1}^2}{\lambda}\|{\boldsymbol\tau}_h^j\|_0^2\\
&\le 2 \frac{c_{-1}^2}{\lambda}\biggl(
\|\bv_t^{j+1}-({\mbf s_h})_t^{j+1}\|_0^2+\Delta t\int_{t_j}^{t_{j+1}}
\|({\mbf s}_h)_{tt}\|_0^2~dt\biggr),
\end{split}
\end{equation}
which allow us to write
\begin{equation}\label{eq:er8}
\Delta t\sum_{j=0}^{n-1}\|P_{V_h}{\boldsymbol\tau}_h^j\|_{-1}^2
\le 2c_{-1}^2\biggl(\frac{\Delta t^2}{\lambda} \int_{0}^{{t_n}}
\|({\mbf s}_h)_{tt}\|_0^2~dt +\frac{t_n}{\lambda} \max_{t_1\le t\le t_n}
\|\bv_t(t)-({\mbf s_h})_t(t)\|_0^2\biggr).
\end{equation}
Thus, inserting (\ref{eq:er13}) and (\ref{eq:er8}) in (\ref{eq:err_evol_a}) and taking $n_0=0$ it follows that
\begin{align*}\label{eq:err_evol_b}
\|\tilde\be_h^{n}\|_0^2+\sum_{j=0}^{n-1}\|\tilde\be_h^{j+1}-\tilde\be_h^j\|_0^2&
+\Delta t\sum_{j=0}^{n-1}\bigl(\nu\|\nabla\tilde\be_h^{j+1}\|_0^2+{\delta}\|\nabla r_h^{j+1}\|_0^2\bigr)
\nonumber\\
\le&c_0\left(\|\tilde\be_h^{0}\|_0^2+{\Delta t}^2\int_0^{t_n} \Bigl(\frac{c_{-1}^2}{\nu\lambda}\|({\mbf s}_h)_{tt}\|_0^2
+\delta\|\nabla (z_h)_t\|_0^2\Bigr)\,dt\right.
\nonumber\\
&{}\left.+c_{-1}^2\frac{t_n}{\nu\lambda}\max_{t_1\le t\le t_n}
\|\bv_t(t)-({\mbf s_h})_t(t)\|_0^2\right).
\end{align*}
Now, in view of~(\ref{eq:error_steady_simplified}) we can write
\begin{equation}\label{eq:er9}
\begin{split}
\|\tilde\be_h^{n}\|_0^2&+\sum_{j=0}^{n-1}\|\tilde\be_h^{j+1}-\tilde\be_h^j\|_0^2
+\Delta t\sum_{j=0}^{n-1}\bigl(\nu\|\nabla\tilde\be_h^{j+1}\|_0^2+{\delta}\|\nabla r_h^{j+1}\|_0^2\bigr)
\\
&\le c_0\left(\|\tilde\be_h^{0}\|_0^2+{\Delta t}^2\int_0^{t_n} \Bigl(\frac{c_{-1}^2}{\nu\lambda}\|({\mbf s}_h)_{tt}\|_0^2+\delta\|\nabla (z_h)_t\|_0^2\Bigr)\,dt\right.\\
&\quad+\frac{C}{\nu^{3}\lambda}t_n\Bigl( \nu^2 h^{2k+2}\max_{t_1\le t\le t_n}\|\bv_t(t)\|_{k+1}^2+C h^4\max_{t_1\le t\le t_n}\| q_t(t)\|_1^2\Bigr)\\
&\quad+\left.\frac{C}{\nu\lambda}t_n \delta^2\max_{t_1\le t\le t_n}\|q_t(t)\|_{1}^2\right).
\end{split}
\end{equation}
Taking $k=1$ and $l=0$ in (\ref{eq:er9}), applying triangle inequality and the error bounds (\ref{eq:error_steady_simplified_linear}) we conclude
(\ref{cota_th1_l2}) and (\ref{cota_th1_H1}).
\end{proof}
\begin{remark}
We observe that the norms $\|({\mbf s}_h)_{tt}\|_0$ and $\delta^{1/2}\|\nabla (z_h)_t\|_0$ in (\ref{laC1}) can be easily bounded in terms of
$\|\bv_{tt}\|_1$ and $\|q_t\|_1$ by adding and subtracting $\bv_{tt}$ and $\nabla (q_t)$, respectively, and applying (\ref{eq:error_steady_simplified}).
\end{remark}

\begin{remark}
Let us observe that taking $\delta=\Delta t$ the analysis above applies to  the standard Euler non-incremental scheme assuming
\begin{equation}\label{eq:cond_pidomask}
\frac{1}{\nu\rho_1^2} h^2\le \Delta t.
\end{equation}
This result is in agreement with the
error bounds in \cite{badia-codina} where the authors prove error bounds for the Euler non-incremental scheme for inf-sup stable elements assuming $\Delta t\ge C h^2$, see \cite[Assumption 7]{badia-codina}. It is also in agreement with the classical results for the continuous in space Euler non-incremental method (see for example \cite{Guermond_overview}) since for $\delta=\Delta t$ the rate of convergence in terms of $\Delta t$ in the $L^2$ norm of the velocity is one and the rate of convergence in the $H^1$ norm of the velocity and the $L^2$ norm of the pressure is one half, see (\ref{cota_th1_l2})-(\ref{cota_th1_H1}).

Let us also observe that condition (\ref{eq:cond_pidomask}) is stronger than condition (\ref{eq:cond_delta2}),
 $\Delta t \le \delta$.  As a consequence, the modified Euler non-incremental scheme with $\delta$  different from $\Delta t$ would be advisable if one wants to use the method for $\Delta t\rightarrow 0$ since there is no need in the modified method to impose  (\ref{eq:cond_pidomask}) for the
time step $\Delta t$.  Moreover, the error analysis carried out explains the instabilities that can be observed in the approximate pressures computed with the standard Euler non-incremental scheme for a fixed $h$ and  $\Delta t$ tending to zero in case of using non inf-sup stable elements, see for example \cite{codina}. In that case, the lower bound in (\ref{eq:cond_pidomask}) is not satisfied and the stability for the pressure induced by equation (\ref{eq:eu_non_tilde2}) disappears. This is in agreement with the analogies stated in Remark 1 between the Euler non-incremental scheme and the PSPG method.
\end{remark}

\begin{remark} It must be observed that the time step restricition~(\ref{eq:cond_delta2}) is not an artifact of the proof but, as it can be easily checked in practice, the modified Euler non-incremental method becomes unstable if $\Delta t$
is taken larger than $2\delta$.

\end{remark}

We now turn to estimate the error in the pressure. We first notice that
we already have an estimate of the form
$$
\Delta t\sum_{j=1}^{n}\delta\|\nabla (q_h^{j}-q(t_j))\|_0^2=\mathcal{O}(\Delta t^2+h^2+\delta)
$$
from~(\ref{cota_th1_H1}).  However, we will obtain error bounds for stronger norms than this one.
\begin{lema}
Under the assumptions of Theorem~\ref{Th1}
the following bound holds
\begin{equation}\label{cota_pre_delta_1}
\begin{split}
\delta \|\nabla (q_h^n-q(t_n))\|_0^2&\le C\Bigl(\nu\|\nabla\tilde \be_h^0\|_0^2+\delta\|\nabla r_h^0\|_0^2+C_4^n t_n \Delta t^2+C_5^n t_n h^4+C_3^n t_n \delta\\
&\quad+\frac{h^2}{\nu}
\left(\nu^2\|\bv(t_n)\|_2^2+\|q(t_n)\|_1^2\right)+\delta\|q(t_n)\|_1^2 \Bigr),
\end{split}
\end{equation}
where $C_3^n$ is the constant in (\ref{laC3}) and
\begin{eqnarray}
C_4^n&=&\max_{0\le t\le t_n}\|({\mbf s}_h)_{tt}(t)\|_0^2+\max_{0\le t\le t_n}\|\nabla (z_h)_t(t)\|_0^2\label{laC4}\\
C_5^n&=&\frac{1}{\nu^2}\left(\nu^2\max_{t_1\le t\le t_n}\|\bv_t(t)\|_2^2+\max_{t_1\le t\le t_n}\|q_t\|_1^2\right).\label{laC5}
\end{eqnarray}
\end{lema}
\begin{proof}
We apply~(\ref{eq:stab_evol_2}) to~(\ref{eq:er1_orig})-(\ref{eq:er3_orig})
so that we get
$$\delta\left\|\nabla r_h^n\right\|_0^2 \le c_0\left(
\nu\|\nabla \tilde\be_h^{0}\|_0^2+\delta\|\nabla r_h^{0}\|_0^2
+ \Delta t\sum_{j=0}^{n-1}(\|P_{V_h}{\boldsymbol\tau}_h^j\|_0^2+\|\nabla (z_h^{j+1}-
z_h^j)\|_0^2\right).
$$
In view of~(\ref{eq:er13}) and~(\ref{eq:er8-1})-(\ref{eq:er8}) we have
\begin{align}
\label{eq:press_1}
\delta\left\|\nabla r_h^n\right\|_0^2 \le&C\left(
\nu\|\nabla \tilde\be_h^{0}\|_0^2+\delta\|\nabla r_h^{0}\|_0^2
+\Delta t^2\int_0^{t_n} \Bigl(\|({\mbf s}_h)_{tt}\|_0^2
+\|\nabla (z_h)_t\|_0^2\Bigr)\,dt\right.
\nonumber\\
&{}\left.
+{t_n}\max_{t_1\le t\le t_n}
\|\bv_t(t)-({\mbf s_h})_t(t)\|_0^2\right),
\end{align}
which in view of~(\ref{eq:error_steady_simplified_linear}) can be written
as
\begin{equation}\label{eq:pressmenos2}
\begin{split}
\delta\left\|\nabla r_h^n\right\|_0^2 \le&C\left(
{\nu}\|\nabla \tilde\be_h^{0}\|_0^2+\delta\|\nabla r_h^{0}\|_0^2
+{\Delta t^2}\int_0^{t_n} \Bigl(\|({\mbf s}_h)_{tt}\|_0^2
+\|\nabla (z_h)_t\|_0^2\Bigr)\,dt\right.
\\
&{}\left.+t_n\frac{h^4}{\nu^2}\Bigl(\nu^2\max_{t_1\le t\le t_n}\|\bv_t(t)\|_{2}^2+\max_{t_1\le t\le t_n}\|q_t(t)\|_{1}^2
\Bigr)+t_n\delta \max_{t_1\le t\le t_n}\|q_t(t)\|_{1}^2\right).
\end{split}
\end{equation}
To conclude we apply the triangle inequality together with ~(\ref{eq:error_steady_simplified_linear}).
\end{proof}
\begin{remark}
The norm $\|\nabla (z_h)_t\|_0$ in the constant $C_4^n$ in (\ref{laC3}) can be bounded as follows. Using inverse inequality (\ref{inv}) and
(\ref{eq:estabi_eli}) we get
\begin{eqnarray*}
\|\nabla (z_h)_t\|_0\le \|\nabla( (z_h)_t-J_h q_t)\|_0+\|\nabla q_t\|_0\le c_{\rm inv}h^{-1}\|(z_h)_t-J_h q_t\|_0+C\|q_t\|_1.
\end{eqnarray*}
Applying now (\ref{eq:error_steady_simplified_linear}) and (\ref{eq:cota_ellip}) we finally bound $\|\nabla (z_h)_t\|_0$ in terms of $\|\bv_t\|_2$ and $\|q_t\|_1$.
\end{remark}
\begin{remark}
As before, the bound (\ref{cota_pre_delta_1}) applies to the standard Euler non-incremental scheme with $\delta=\Delta t$ assuming in that case $ h^2/(\nu\rho_1^2)\le \Delta t$, i.e. condition (\ref{eq:cond_pidomask}) holds.
We can deduce from (\ref{cota_pre_delta_1}) that the errors in the pressure are bounded in terms of $\delta \|\nabla r_h^0\|_0^2$.
Let us observe that using (\ref{eq:eu_non_tilde}) to get $\tilde \bv_h^1$, apart from the standard initial condition for the velocity $\tilde \bv_h^0$ one would need
an initial pressure $q_h^0$. If one takes for example $q_h^0=0$, then one gets $\|\nabla r_h^0\|_0^2=\|\nabla z_h^0\|_0$, the last norm being $O(1)$ as can be proved arguing as in Remark~3.   Then $\delta \|\nabla r_h^0\|_0^2=O(\delta)$ which is of the same order as the last term in (\ref{cota_pre_delta_1}). As a consequence, the choice $q_h^0=0$ in (\ref{eq:eu_non_tilde}) does not spoil the rate of convergence of the pressure.
\end{remark}

Next lemma gets an improvement of the error bound (\ref{eq:pressmenos2}) that will allow us to understand the effect of the initial condition chosen on the error in the approximate pressure.
\begin{lema} Let $r_h^n=q_h^n-z_h^n$ the error defined in (\ref{eq:defer}). Under the assumptions of Theorem~\ref{Th1} the following bound holds
\begin{equation}\label{eq:press_2tmenos1}
\|\nabla r_h^n\|_0^2 \le C\Bigl (\frac{\|\tilde \be_h^0\|_0^2}{\delta t_n}+\frac{\Delta t^2}{t_n^2}\|\nabla r_h^0\|_0^2+
\Delta t\bigl (C_1^n+t_nC_4^n\bigr)+\delta C_3^n\bigl((\nu\lambda)^{-1} +t_n \bigr)+\nu h^2(C_2^n+t_nC_5^n)\Bigr ),
\end{equation}
where $C_1^n$, $C_2^n$, $C_3^n$, $C_4^n$ and $C_5^n$ are the constants in (\ref{laC1}), (\ref{laC2}), (\ref{laC3}), (\ref{laC4}) and (\ref{laC5}) respectively.
\end{lema}
\begin{proof}
Multiply~(\ref{eq:er1_orig}) and~(\ref{eq:er3_orig}) by $t_{n+1}$, and add $\pm (t_n\tilde\be_n/\Delta t,{\bs \chi}_h)$ and~$\pm t_n(\nabla r_h^n,{\bs \chi}_h)$ to~(\ref{eq:er1_orig}), so that
for
$$
\bw_h^n=t_n\tilde\be_h^n,\qquad y_h^n=t_nr_h^n,
\qquad {\bb_h^n}=t_{n+1}P_{V_h}{\boldsymbol\tau}_h^n + \tilde \be_h^n,
$$
and
$$
d_n=t_{n+1}(z_h^{n+1}-z_h^n) - \Delta t r_h^n,
$$
we get~(\ref{eq:er1})-(\ref{eq:er3}). Applying~(\ref{eq:stab_evol_2})
we have
\begin{equation}\label{eq:press_1t}
\begin{split}
\delta t_n^2\left\|\nabla r_h^n\right\|_0^2 \le &
 c_0\Delta t\sum_{j=0}^{n-1}t_{j+1}^2(\|{\boldsymbol\tau}_h^j\|_0^2+\|\nabla (z_h^{j+1}-
z_h^j)\|_0^2)\\
&+c_0\Delta t\sum_{j=0}^{n-1}(\|\tilde \be_h^j\|_0^2 +{\Delta t}^2\|\nabla r_h^j\|_0^2).
\end{split}
\end{equation}
For the second sum on the right hand side above using (\ref{eq:cond_delta2}) and $\Delta t\le t_n$ for $n\ge1$ we get
\begin{equation*}
\Delta t\sum_{j=0}^{n-1}(\|\tilde \be_h^j\|_0^2 +\Delta t^2\|\nabla r_h^j\|_0^2)
\le t_n \max_{0\le j\le n-1}\|\tilde\be_h^j\|_0^2 + \delta \Delta t^2\|\nabla r_h^0\|_0^2+
t_n\sum_{j=1}^{n-1}\delta\Delta t\|\nabla r_h^j\|_0^2
\end{equation*}
and then apply~(\ref{eq:er9}) to reach
\begin{equation}\label{eq:extra1}
\begin{split}
\Delta t\sum_{j=0}^{n-1}(\|\tilde \be_h^j\|_0^2 &+\Delta t^2\|\nabla r_h^j\|_0^2)\\
&\le C t_n \|\tilde \be_h^0\|_0^2+\delta \Delta t^2\|\nabla r_h^0\|_0^2
+C\bigl(C_1^n t_n^2\Delta t^2+C_2^n t_n^2 h^4+\frac{C_3^n}{\nu\lambda}t_n^2\delta^2\bigr),
\end{split}
\end{equation}
where $C_1^n$, $C_2^n$ and $C_3^n$ are the constants in (\ref{laC1}), (\ref{laC2}) and (\ref{laC3}) respectively.

For the first sum on the
right hand side of~(\ref{eq:press_1t}) we write $t_{j+1}^2 \le t_n^2$ and
apply~(\ref{eq:preer13}) and (\ref{eq:er8-1}), (\ref{eq:er8}). Then, we get
\begin{equation*}
\begin{split}
\Delta t\sum_{j=0}^{n-1}t_{j+1}^2(\|{\boldsymbol\tau}_h^j\|_0^2&+\|\nabla (z_h^{j+1}-
z_h^j)\|_0^2\\
&\le Ct_n^2\Delta t^2\left(\int_0^{t_n}\left(\|({\mbf s}_h)_{tt}(t)\|_0^2+\|\nabla(z_h)_t(t)\|_0^2\right)dt\right)\\
&{}\quad+C t_n^3\max_{t_1\le t\le t_n}\|\bv_t(t)-({\mbf s}_h)_t(t)\|_0^2,
\end{split}
\end{equation*}
and applying (\ref{eq:error_steady_simplified_linear})
\begin{equation}\label{eq:extra2}
\Delta t\sum_{j=0}^{n-1}t_{j+1}^2\bigl (\|{\boldsymbol\tau}_h^j\|_0^2+\|\nabla (z_h^{j+1}-
z_h^j)\|_0^2\bigr )\le C\left(C_4^nt_n^3\Delta t^2+C_5^n t_n^3 h^4+C_3^n t_n^3 \delta^2\right),
\end{equation}
where $C_3^n$, $C_4^n$ and $C_5^n$ are the constants in (\ref{laC3}), (\ref{laC4}) and (\ref{laC5}).
Inserting (\ref{eq:extra1}) and (\ref{eq:extra2}) into (\ref{eq:press_1t}) we obtain
\begin{eqnarray*}
\delta\|\nabla r_h^n\|_0^2 &\le& C\left(\frac{\|\tilde \be_h^0\|_0^2}{t_n}+\frac{\delta\Delta t^2}{t_n^2}\|\nabla r_h^0\|_0^2+
\Delta t^2\left(C_1^n +t_n C_4^n\right))+h^4(C_2^n+t_nC_5^n)\right)\nonumber\\
&&\quad+C\delta^2C_3^n\bigl((\nu\lambda)^{-1}+t_n\bigr).
\end{eqnarray*}
Dividing by $\delta$ and using conditions (\ref{eq:cond_delta}) and (\ref{eq:cond_delta2}) we reach (\ref{eq:press_2tmenos1}).
\end{proof}
\begin{remark}\label{re:remark6}
Let us assume we choose the initial condition for the velocity such that the error $\|\tilde\be_h^{0}\|_0=O(h^2)$ and
$q_h^0=0$. Then  $\|\nabla r_h^0\|_0=O(1)$ (see Remark 7) and, as a consequence, the second term in (\ref{eq:press_2tmenos1}) for $n=1$ is $O(1)$ and the first one is $O(h^2/\Delta t)$ and then is also $O(1)$ in case (\ref{eq:cond_pidomask}) is satisfied or in can be worse than $O(1)$ if we consider the modified Euler non-incremental scheme and we take $\Delta t$ tending to 0 for a fixed $h$.

 However, in the case $(\tilde \bv_h^0,r_h^0)=({\mbf s}_h^0,z_h^0)$, i.e., taking as initial approximation to the velocity and pressure the stabilized Stokes approximation of the solution $(\bv,p)$ of (\ref{eq:evo_stokes}) at time $t=0$, as suggested in \cite{John_Novo_PSPG},  the errors  $\|\nabla r_h^n\|_0$ are $O(h)$ from the first step.  This result is in agreement with both theoretical and  numerical results shown in \cite{John_Novo_PSPG} for the PSPG method applied
to the evolutionary Stokes equations and supports the  analogy between the Euler non incremental projection scheme and the PSPG method previously found in the literature \cite{Guermond_Quar_IJNMF}, \cite{ranacher}. We refer the reader  to \cite{John_Novo_PSPG} for details about the practical computation of the initial stabilized Stokes approximation using only the given data ${\mbf g}$ and $\bv^0$.
\end{remark}
\begin{lema} Under the assumptions of Lemma~4 and assuming $(\tilde \bv_h^0,q_h^0)=({\mbf s}_h^0,z_h^0)$, the following bound holds for the error $q_h^n-q(t_n)$
\begin{equation}\label{eq:press_L2tn}
\begin{split}
\|q_h^n-q(t_n)\|_0^2 &\le C\lambda^{-1}\left(\Delta t\bigl (C_1^n+t_nC_4^n\bigr)+\delta C_3^n\bigl((\nu\lambda)^{-1} +t_n \bigr)+\nu h^2(C_2^n+t_nC_5^n)\right)\\
&\quad +C h^2(\nu^2\|\bv(t_n)\|_2^2+\|q(t_n)\|_1^2)+C\nu\delta\|q(t_n)\|_1^2,
\end{split}
\end{equation}
where $C_1^n$, $C_2^n$, $C_3^n$, $C_4^n$ and $C_5^n$ are the constants in (\ref{laC1}), (\ref{laC2}), (\ref{laC3}), (\ref{laC4}) and (\ref{laC5}) respectively.
\end{lema}
\begin{proof} Applying Poincar\'e inequality and (\ref{eq:press_2tmenos1}) we get
$$
\|r_h^n\|_0^2\le C\lambda^{-1}\|\nabla r_h^n\|_0^2\le C\lambda^{-1}\left(\Delta t\bigl (C_1^n+t_nC_4^n\bigr)+\delta C_3^n\bigl((\nu\lambda)^{-1} +t_n \bigr)+\nu h^2(C_2^n+t_nC_5^n)\right).
$$
Now, (\ref{eq:press_L2tn}) follows applying triangle inequality together with (\ref{eq:error_steady_simplified_linear}).
\end{proof}

To conclude this section we get an error bound for the pressure valid for any initial condition.
\begin{theorem}\label{Th2} Under the assumptions of Theorem~\ref{Th1}
 the following bound holds
\begin{equation}\label{eq:cotapresionl2final}
\begin{split}
\sum_{j=1}^n \Delta t \|q_h^j-q(t_j)\|_0^2&
\le C (t_n\nu+\lambda^{-1})(\nu\|\nabla \tilde \be_h^0\|_0^2+ \delta\|\nabla r_h^0\|_0^2)
+C\nu\|\tilde\be_h^0\|_0^2+C t_n^2\nu C_3^n \delta\\
&\quad +C t_n h^2\max_{t_1\le t\le t_n}(\nu\|\bv(t)\|_2^2+\|q(t)\|_1^2)+C t_n \nu\delta \max_{t_1\le t\le t_n}\|q(t)\|_1^2\\
&\quad+Ct_{n+1}\Delta t^2\left(\nu t_n C_5^n+\nu C_1^{n+1}+C_6^{n+1}+\lambda^{-1}C_7^{n+1}\right)\\
&\quad+Ct_{n+1}h^4\left(\nu t_n C_4^n+\nu C_2^{n+1}+\lambda^{-1}C_5^{n+1}\right)+Ct_{n+1}\lambda^{-1}C_3^{n+1}\delta^2.
\end{split}
\end{equation}
 where
 \begin{eqnarray}\label{laC6}
 C_6^n&=&\max_{t_1\le t\le t_n}\|(z_h)_t(t)\|_0^2+\lambda^{-1}\max_{t_1\le t\le t_n}\|\nabla (z_h)_t\|_0^2,\\
 C_7^n&=&\max_{t_1\le t\le t_n}\|({\mbf s}_h)_{tt}(t)\|_0^2,\label{laC7}
 \end{eqnarray}
 and $C_1^n$, $C_2^n$, $C_3^n$, $C_4^n$, $C_5^n$, $C_6^n$ and $C_7^n$ are the constants in (\ref{laC1}), (\ref{laC2}), (\ref{laC3}), (\ref{laC4}), (\ref{laC5}), (\ref{laC6}) and (\ref{laC7}).
\end{theorem}
\begin{proof}
We first observe that from (\ref{eq:pressmenos2}) we get
\begin{equation}\label{eq:press_2tmenos1_bis}
\delta\left\|\nabla r_h^n\right\|_0^2 \le C\left(\nu\|\nabla \tilde \be_h^0\|_0^2
+\delta \|\nabla r_h^0\|_0^2+C_4^n t_n \Delta t^2+C_5^n t_n h^4+C_3^n t_n\delta\right),
\end{equation}
where $C_3^n$, $C_4^n$ and $C_5^n$ are the constants in (\ref{laC3}), (\ref{laC4}) and (\ref{laC5}).
Applying Lemma~\ref{lema_presion} we get
$$
\|r_h^n\|_0\le C\nu^{1/2}\delta^{1/2}\|\nabla r_h^n\|_0+C\sup_{{\bs\chi}_h\in V_h}\frac{(r_h,\nabla \cdot {\bs\chi}_h)}{\|{\bs\chi}_h\|_1}.
$$
From (\ref{eq:er1_orig}) we obtain
\begin{equation*}
\sup_{{\bs\chi}_h\in V_h}\frac{(r_h,\nabla \cdot {\bs\chi}_h)}{\|{\bs\chi}_h\|_1}\le \left\|\frac{\tilde \be_h^{n+1}-\tilde \be_h^n}{\Delta t}\right\|_{-1}+\nu\|\nabla \tilde \be_h^{n+1}\|_0+\|P_{V_h}\hat {\boldsymbol\tau}_h^n\|_{-1}+\|z_h^n-z_h^{n+1}\|_0.
\end{equation*}
Then, we can write
\begin{equation}\label{eq:cotaL2varios2}
\begin{split}
\sum_{j=1}^n\Delta t\|r_h^j\|_0^2&\le C\nu\sum_{j=1}^n\Delta t\delta\|\nabla r_h^j\|_0^2+C\sum_{j=1}^n\Delta t\left\|\frac{\tilde \be_h^{j+1}-\tilde \be_h^j}{\Delta t}\right\|_{-1}^2\\
&\quad+C \nu^2\sum_{j=1}^n\Delta t\|\nabla \tilde \be_h^{j+1}\|_0^2+C \sum_{j=1}^n\Delta t \|P_{V_h}\hat {\boldsymbol\tau}_h^j\|_{-1}^2\\
&\quad+C \sum_{j=1}^n\Delta t\|z_h^j-z_h^{j+1}\|_0^2.
\end{split}
\end{equation}
To bound the first term on the right-hand side of (\ref{eq:cotaL2varios2}) we apply (\ref{eq:press_2tmenos1_bis}) and get
\begin{equation}\label{eq:prel21}
\nu\sum_{j=1}^n\Delta t\delta\|\nabla r_h^j\|_0^2 \le Ct_n\nu\left(\nu\|\nabla \tilde \be_h^0\|_0^2
+\delta \|\nabla r_h^0\|_0^2+t_n\left(C_4^n  \Delta t^2+C_5^n  h^4+C_3^n \delta\right)\right).
\end{equation}
 For the third term we apply (\ref{eq:er9}) with $n+1$ instead of $n$ to obtain
 \begin{equation}\label{eq:prel22}
\nu^2\sum_{j=1}^n\Delta t\|\nabla \tilde \be_h^{j+1}\|_0^2
\le C\nu\left(\|\tilde \be_h^0\|_0^2
+ t_{n+1}C_1^{n+1}\Delta t^2+ t_{n+1}C_2^{n+1}h^4+C_3^{n+1}(\nu\lambda)^{-1} t_{n+1} \delta^2\right).
\end{equation}
Applying (\ref{eq:er8}) with $n$ replaced by $n+1$ again to bound the forth term we get
\begin{equation}\label{eq:prel23}
\sum_{j=1}^n\Delta t \|P_{V_h}\hat {\boldsymbol\tau}_h^j\|_{-1}^2\le C\lambda^{-1}\left( t_{n+1}C_7^{n+1}\Delta t^2+ t_{n+1}C_5^{n+1}h^4+C_3^{n+1} t_{n+1} \delta^2\right).
\end{equation}
For the last term on the right-hand side of (\ref{eq:cotaL2varios2}) we observe that
\begin{equation}\label{eq:prel24}
\sum_{j=1}^n\Delta t\|z_h^j-z_h^{j+1}\|_0^2\le \Delta t^2\int_{t_1}^{t_{n+1}}\|(z_h)_t\|_0^2~dt.
\end{equation}
To conclude we will bound the second term on the right-hand side of (\ref{eq:cotaL2varios2}) applying (\ref{eq:cota_menos1}) and (\ref{eq:stab_evol_2}).
\begin{equation*}
\begin{split}
\sum_{j=1}^n\Delta t\left\|\frac{\tilde \be_h^{j+1}-\tilde \be_h^j}{\Delta t}\right\|_{-1}^2&\le c_{-1}^2\lambda^{-1}\sum_{j=0}^n\Delta t\left\|\frac{\tilde \be_h^{j+1}-\tilde \be_h^j}{\Delta t}\right\|_{0}^2\\
&\le c_{-1}^2c_0\lambda^{-1}\left(\nu\|\nabla \tilde \be_h^0\|_0^2+\delta \|\nabla r_h^0\|_0^2\right)\\
&\quad+c_{-1}^2c_0\lambda^{-1}\Bigl (\Delta t\sum_{j=0}^n\|{\boldsymbol\tau}_h^j\|_0^2+\|\nabla (z_h^j-z_h^{j+1})\|_0^2\Bigr)
\end{split}
\end{equation*}
To bound the last two terms on the right-hand side above we apply (\ref{eq:er8}) for the first one as before and argue as usual for the second so that we reach
\begin{equation}\label{eq:prel25}
\begin{split}
\sum_{j=1}^n\Delta t\left\|\frac{\tilde \be_h^{j+1}-\tilde \be_h^j}{\Delta t}\right\|_{-1}^2&\le C\lambda^{-1}\left(\nu\|\nabla \tilde \be_h^0\|_0^2+\delta \|\nabla r_h^0\|_0^2\right)\\
&\quad+ C\lambda^{-1}\left( t_{n+1}C_7^{n+1}\Delta t^2+ t_{n+1}C_5^{n+1}h^4+C_3^{n+1} t_{n+1} \delta^2\right)\\
&\quad+C\lambda^{-1}\Delta t^2\int_{t_1}^{t_{n+1}}\|\nabla (z_h)_t\|_0^2~dt.
\end{split}
\end{equation}
Inserting (\ref{eq:prel21}), (\ref{eq:prel22}), (\ref{eq:prel23}), (\ref{eq:prel24}) and (\ref{eq:prel25}) into (\ref{eq:cotaL2varios2}) we obtain
\begin{eqnarray}\label{eq:prelast}
\sum_{j=1}^n \Delta t \|r_h^j\|_0^2&\le& C (t_n\nu
+\lambda^{-1})(\nu\|\nabla \tilde \be_h^0\|_0^2+ \delta\|\nabla r_h^0\|_0^2)
+C\nu\|\tilde\be_h^0\|_0^2+C t_n^2\nu C_3^n\delta\\
&&+Ct_{n+1}\Delta t^2\left(\nu t_n C_4^n+\nu C_1^{n+1}+C_6^{n+1}+\lambda^{-1}C_7^{n+1}\right)+Ct_{n+1}\lambda^{-1}C_3^{n+1}\delta^2\nonumber\\
&&+Ct_{n+1}h^4\left(\nu t_n C_5^n+\nu C_2^{n+1}+\lambda^{-1}C_5^{n+1})\right).\nonumber
\end{eqnarray}
 Using the triangle inequality together with (\ref{eq:error_steady_simplified_linear}) we finally reach (\ref{eq:cotapresionl2final}).
 \end{proof}
\begin{remark}
We observe that  Remark 5 can be applied to the error bound (\ref{eq:cotapresionl2final}). On the one hand, the error bound for the pressure holds for the standard Euler non-incremental scheme whenever $\Delta t$ satisfies (\ref{eq:cond_pidomask}). However, for the modified Euler non-incremental scheme only condition (\ref{eq:cond_delta}) is required so that for any $h$ we can allow $\Delta t\rightarrow 0$ without loosing the optimal rate of convergence. On the other hand, any initial approximation for the velocity satisfying $\|\nabla \tilde \be_h^0\|_0=\mathcal{O}(h)$ and any initial pressure
satisfying~$\| \nabla r_h^0\|_0=\mathcal{O}(1)$ (including $q_h^0=0$) will result in an optimal error bound of size
$\mathcal{O}(h+\Delta t+\delta^{1/2})$ for the discrete $L^2$ norm of the pressure error.
\end{remark}
\bigskip
\subsection{Euler incremental scheme}
Let us  denote by $(\bv_h^n,\tilde \bv_h^n,q_h^n)$, $n=1,2,\ldots,$\  $\tilde \bv_h^n\in V_h$, $q_h^n\in Q_h$ and $ \bv_h^n\in V_h+\nabla Q_h$
the approximations to the velocity and pressure at time $t_n=n\Delta t$, $\Delta t=T/N$, $N>0$ obtained with the following modified Euler incremental scheme
\begin{eqnarray}\label{eq:eu_inc0}
&&\left(\frac{\tilde \bv_h^{n+1}-\bv_h^n}{\Delta t},{\bs\chi}_h\right)+\nu(\nabla \tilde \bv_h^{n+1},\nabla {\bs\chi}_h)+(\nabla q_h^n,\bs\chi_h)=({\mbf g}^{n+1},{\bs\chi}_h),\quad \forall {\bs\chi}_h\in V_h\nonumber\\
&&(\nabla \cdot \tilde \bv_h^{n+1},\psi_h)=-\delta(\nabla (q_h^{n+1}-q_h^n),\nabla \psi_h),\quad \forall \psi_h\in Q_h,\\
&&\bv^{n+1}_h=\tilde \bv_h^{n+1}-\delta\nabla (q_h^{n+1}-q_h^n).\nonumber
\end{eqnarray}
Let us observe that for $\delta=\Delta t$ in (\ref{eq:eu_inc0}) we have the classical  Euler incremental scheme~\cite{Guermond_Quar2}. It is well known that this method is not stable if  non inf-sup stable mixed finite elements are employed~\cite{codina}. Following the suggestion in  \cite{Javier-Volker-Julia} (see also \cite{Prohl}) we consider  the following method
\begin{eqnarray}\label{eq:eu_inc1}
&&\left(\frac{\tilde \bv_h^{n+1}-\bv_h^n}{\Delta t},{\bs\chi}_h\right)+\nu(\nabla \tilde \bv_h^{n+1},\nabla {\bs\chi}_h)+(\nabla q_h^n,\bs\chi_h)=({\mbf g}^{n+1},{\bs\chi}_h),\quad \forall {\bs\chi}_h\in V_h\nonumber\\
&&(\nabla \cdot \tilde \bv_h^{n+1},\psi_h)=-\delta(\nabla (q_h^{n+1}-q_h^n),\nabla \psi_h)-\delta_2(\nabla q_h^{n+1},\nabla \psi_h),\quad \forall \psi_h\in Q_h,\\
&&\bv^{n+1}_h=\tilde \bv_h^{n+1}-\delta\nabla (q_h^{n+1}-q_h^n),\nonumber
\end{eqnarray}
where $\delta_2$ is a second stabilization parameter.

In case $\delta=\Delta t$ we can remove $\bv_h^n$ from (\ref{eq:eu_inc1}) to get
\begin{align}\label{eq:eu_inc_tilde}
&\left(\frac{\tilde \bv_h^{n+1}-\tilde \bv_h^n}{\Delta t},{\bs\chi}_h\right)+\nu(\nabla \tilde \bv_h^{n+1},\nabla {\bs\chi}_h)+(\nabla (2q_h^n-q_h^{n-1}),{\bs \chi}_h)=({\mbf g}^{n+1},{\bs\chi}_h),\,\forall {\bs\chi}_h\in V_h\nonumber\\
&(\nabla \cdot \tilde \bv_h^{n+1},\psi_h)=-\delta(\nabla (q_h^{n+1}-q_h^n),\nabla \psi_h)-\delta_2(\nabla q_h^{n+1},\nabla \psi_h),\quad \forall \psi_h\in Q_h.
\end{align}
As in the previous section the method we study is (\ref{eq:eu_inc_tilde}) with $\delta$ not necessarily equal to $\Delta t$. However, since now  the parameter $\delta_2$ is the one equivalent to the stabilization parameter in the PSPG method a reasonable choice for the parameters would be $\delta=\Delta t$ and $\delta_2$ defined as $\delta$ in (\ref{eq:cond_delta}). In this section we do not carry out the error analysis of the method for these values of the stabilization parameters. We only study the errors in the particular  case $\delta_2=\delta$ defined in (\ref{eq:cond_delta}) since in that case the analysis is a direct consequence of the error analysis of the previous section. The analysis of the Euler non-incremental scheme in time with finite elements in space for inf-sup stable elements can be found in \cite{Guermond_Quar2}. To our knowledge there is no error analysis for this method in case of using non inf-sup stable elements. Some stability estimates can be found in \cite{codina}, but for stabilization more related to local projection stabilization than the one we consider here, which is more related  to PSPG stabilization. In \cite{codina} instead of adding $\delta_2(\nabla q_h^{n+1},\nabla \psi_h)$ as in (\ref{eq:eu_inc_tilde}) the term $\delta_2(\nabla q_h^{n+1}-\pi_h,\nabla \psi_h)$ is added where $\pi_h$ is the projection of $\nabla q_h^{n+1}$ into certain finite element space.

Going back to (\ref{eq:eu_inc_tilde}) we first observe that for $\delta_2=\delta$ and
$$
\hat q_h^n=2q_h^n-q_h^{n-1},
$$
it is easy to check that $(\tilde \bv_h^n,\hat q_h^n)$ satisfies (\ref{eq:eu_non_tilde}) and then we can apply the error bounds (\ref{cota_th1_l2}), (\ref{cota_th1_l2}) and (\ref{eq:cotapresionl2final})  to $(\tilde \bv_h^n,\hat q_h^n)$. To conclude this section we prove an error bound for $q_h^n-q(t_n)$.
\begin{theorem}\label{Th3} Let $(\bv,q)$ be the solution of (\ref{eq:evo_stokes}) and let $(\tilde \bv_h^n,q_h^n)$, $n\ge 1$,  be the solution of (\ref{eq:eu_inc_tilde}). Assume $\delta=\delta_2$ satisfies condition (\ref{eq:cond_delta}) and $\Delta t$ satisfies condition (\ref{eq:cond_delta2}).
Then, the following bounds hold
\begin{align}\label{eq:cotapresionl2finalinc}
\sum_{j=1}^n \Delta t \|q_h^j-q(t_j)\|_0^2&\le C (t_n\nu+\lambda^{-1})(\nu\|\nabla \tilde \be_h^0\|_0^2+ \delta\|\nabla r_h^0\|_0^2)
+C\nu\|\tilde\be_h^0\|_0^2+C t_n^2\nu C_3^n\delta\nonumber\\
&\quad+ C \Delta t\|q_h^0-q(0)\|_0^2\nonumber\\
&\quad +C t_n h^2\max_{t_1\le t\le t_n}(\nu\|\bv(t)\|_2^2+\|q(t)\|_1^2)+C t_n \nu\delta \max_{t_1\le t\le t_n}\|q(t)\|_1^2\nonumber\\
&\quad +Ct_{n+1}\Delta t^2\bigl(\nu t_n C_4^n+\nu C_1^{n+1}+C_6^{n+1}+\lambda^{-1}C_7^{n+1}
+\max_{t_0\le t\le t_n}\|q_t(t)\|_0^2\bigr),\nonumber\\
&\quad +C t_{n+1}h^4(\nu t_nC_5^n+ \nu C_2^{n+1}+\lambda^{-1}C_5^{n+1})+C t_{n+1}C_3^{n+1}\lambda^{-1}\delta^2.
\end{align}
 where $C_1^n$, $C_2^n$, $C_3^n$, $C_4^n$, $C_5^n$, $C_6^n$ and $C_7^n$ are the constants in (\ref{laC1}), (\ref{laC2}), (\ref{laC3}), (\ref{laC4}), (\ref{laC5}), (\ref{laC6}) and (\ref{laC7}) respectively.
\end{theorem}
\begin{proof}
We first observe that
$$
q_h^n-q(t_n)=\frac{1}{2}(\hat q_h^n-q(t_n))+\frac{1}{2}(q_h^{n-1}-q(t_{n-1}))+\frac{1}{2}(q(t_{n-1})-q(t_n)).
$$
Taking into account that $(a+b+c)^2\le 4a^2+4b^2+2c^2$ for any $a,b,c\in {\mathbb R}$ we can write
\begin{eqnarray*}
\sum_{j=1}^n\Delta t \|q_h^j-q(t_j)\|_0^2&\le& \sum_{j=1}^n\Delta t \|\hat q_h^j-q(t_j)\|_0^2+\frac{1}{2}\sum_{j=1}^n\Delta t \|q_h^{j-1}-q(t_{j-1})\|_0^2\nonumber\\
&&\quad +\sum_{j=1}^n\Delta t \|q(t_{j-1})-q(t_j)\|_0^2
\end{eqnarray*}
and then
\begin{eqnarray*}
\frac{1}{2}\sum_{j=1}^n\Delta t \|q_h^j-q(t_j)\|_0^2&\le& \sum_{j=1}^n\Delta t \|\hat q_h^j-q(t_j)\|_0^2+\frac{1}{2}\Delta t \|q_h^{0}-q(0)\|_0^2\nonumber\\
&&\quad +\sum_{j=1}^n\Delta t \|q(t_{j-1})-q(t_j)\|_0^2.
\end{eqnarray*}
Applying (\ref{eq:cotapresionl2final}) and taking into account that
$$
\sum_{j=1}^n\Delta t \|q(t_{j-1})-q(t_j)\|_0^2\le \Delta t^2\int_{t_0}^{t_n}\|q_t\|_0^2~dt
$$
we reach (\ref{eq:cotapresionl2finalinc}).
\end{proof}
\begin{remark}
Choosing $\delta=\delta_2=\Delta t$ the error analysis of the modified Euler non-incremental scheme gives the analysis of the classical Euler non-incremental scheme with PSPG stabilization whenever condition (\ref{eq:cond_pidomask}) is assumed.
\end{remark}
\section{Numerical experiments}

In this section, we take~$\Omega=[0,1]\times[0,1]$ and all grids are regular $N\times N$ triangular grids with SWNE diagonals for different
values of~$N$

We first check that no better than second order convergence is achieved
in the velocity. For this purpose we consider the errors of the
steady state approximation~(\ref{eq:pro_stokes2}) to~(\ref{eq:stokes})
with $\nu=0.01$
where the forcing~term~$\hat\bg$ is such that the solution
is
\begin{align}
\label{eq:motiv_v}
{\mbf s}(x,y)&=\left[\begin{array}{cc} x^2(1-x)^2\sin(2\pi y)\\
-2x(1+3x+2x^2)\sin^2(\pi y)\end{array}\right],\\
z(x,y)&=\sin(x)\cos(y) +(\cos(1)-1)\sin(1).
\label{eq:motiv_p}
\end{align}
This solution is taken from~\cite{Berrone-Marro} and it
is used as a motivating example in~\cite{John_Novo_PSPG}.
We show the errors~${\mbf s}_h-I_h({\mbf s})$ and~$z_h-I_h(z)$, where
$I_h$ denotes the standard (Lagrange) interpolant on $N\times N$ grids, with $N$ ranging from~$20$ to~$320$ in the case of linear elements and
from~$10$ to~$160$ in the case of quadratic elements. Errors in~$L^2$ for the velocity for different values of~$\delta=h^2/(\nu\rho^2)$ are plotted as a function of
the mesh size~$h$ on the left of~Fig.~\ref{fig:steady_v}, where the results corresponding to a given value of~$\rho$ are joined by straight segments
of continuous and discontinuous line for linear and quadratic elements, respectively.
It can be observed that, for small values of $\rho$, linear and quadratic elements
produce the same errors. As $\rho$ increases the errors with
quadratic elements are smaller than those of linear elements but the convergence
rate is two for both methods. We can also observe that the optimal value of $\rho$ for the errors is around $\rho\approx100$ which gives
$\delta\approx 0.01 h^2$. This value is not far away from the value of $\delta=0.005 h^2$ suggested in \cite{John_Novo_PSPG} for the PSPG method.
\begin{figure}
\begin{center}
\includegraphics[height=5.3truecm]{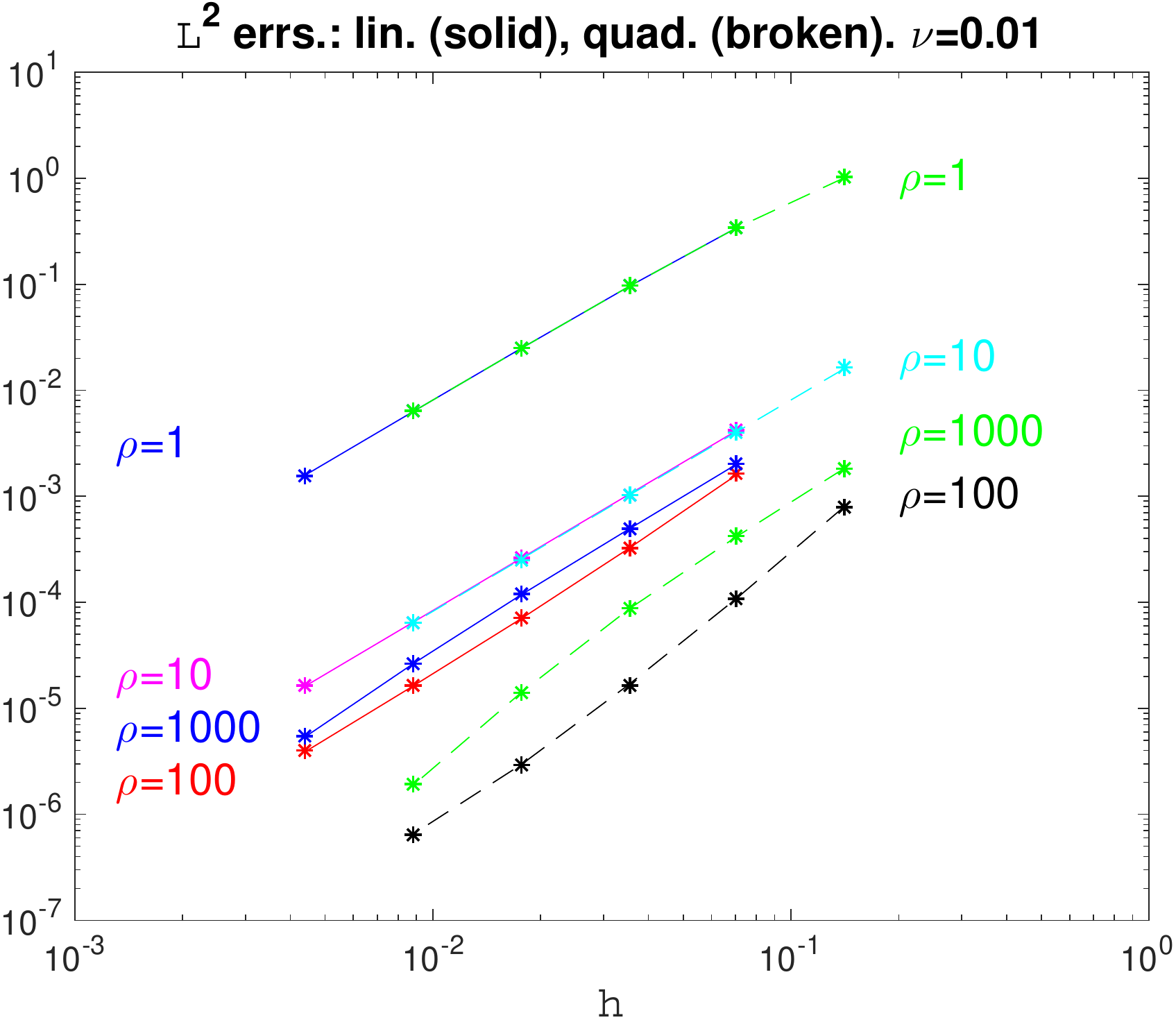}
\quad
\includegraphics[height=5.3truecm]{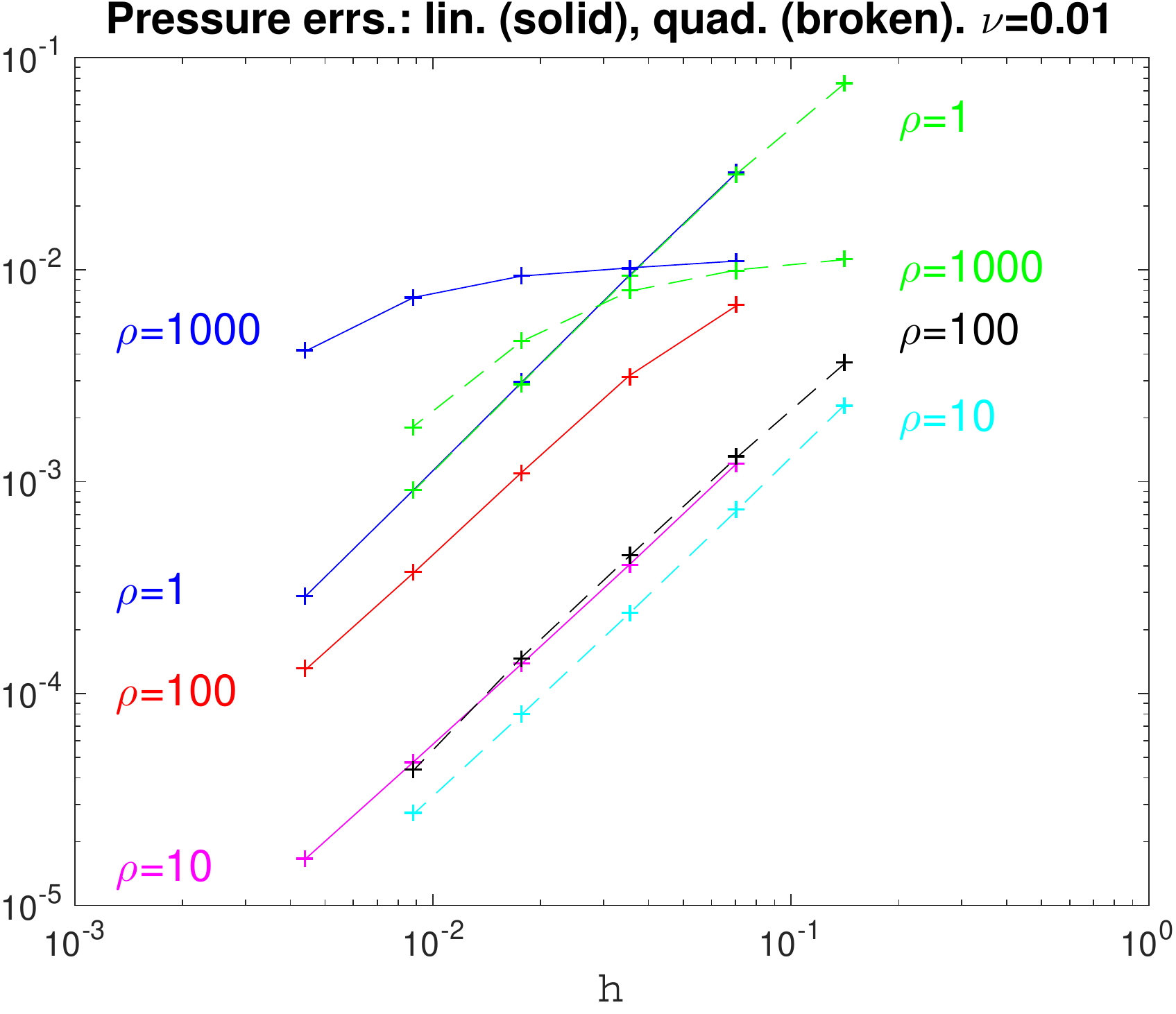}
\end{center}
\caption{On the left, errors ${\mbf s}_h -I_h({\mbf s})$ in $L^2$ for linear (solid line) and quadratic (broken line) elements for different values of~$\delta=h^2/(\nu\rho^2)$. On  the right, pressure errors $z_h-I_h(z)$.}\label{fig:steady_v}
\end{figure}

In the errors for the pressure, shown on the right of~Fig.~\ref{fig:steady_v} we can observe that for $\rho=1$ linear and quadratic elements produce the same errors but as $\rho$ increases the errors of quadratic elements are smaller although with the same convergence rate than linear elements. For the pressure the best value of $\rho$ is around $\rho\approx 10$ and as $\rho$ increases the errors in the pressure increase remarkably.  For $\rho=1000$ (i.e., $\delta=0.0001h^2$) we can observe that the errors of the pressure hardly decrease for most of the largest values~$h$. This result is in agreement with the fact that
$\delta$ must be strictly positive to stabilize
the pressure in~(\ref{eq:pro_stokes2}) if non inf-sup stable mixed finite elements are used.

\begin{figure}[h]
\begin{center}
   \includegraphics[height=5.3truecm]{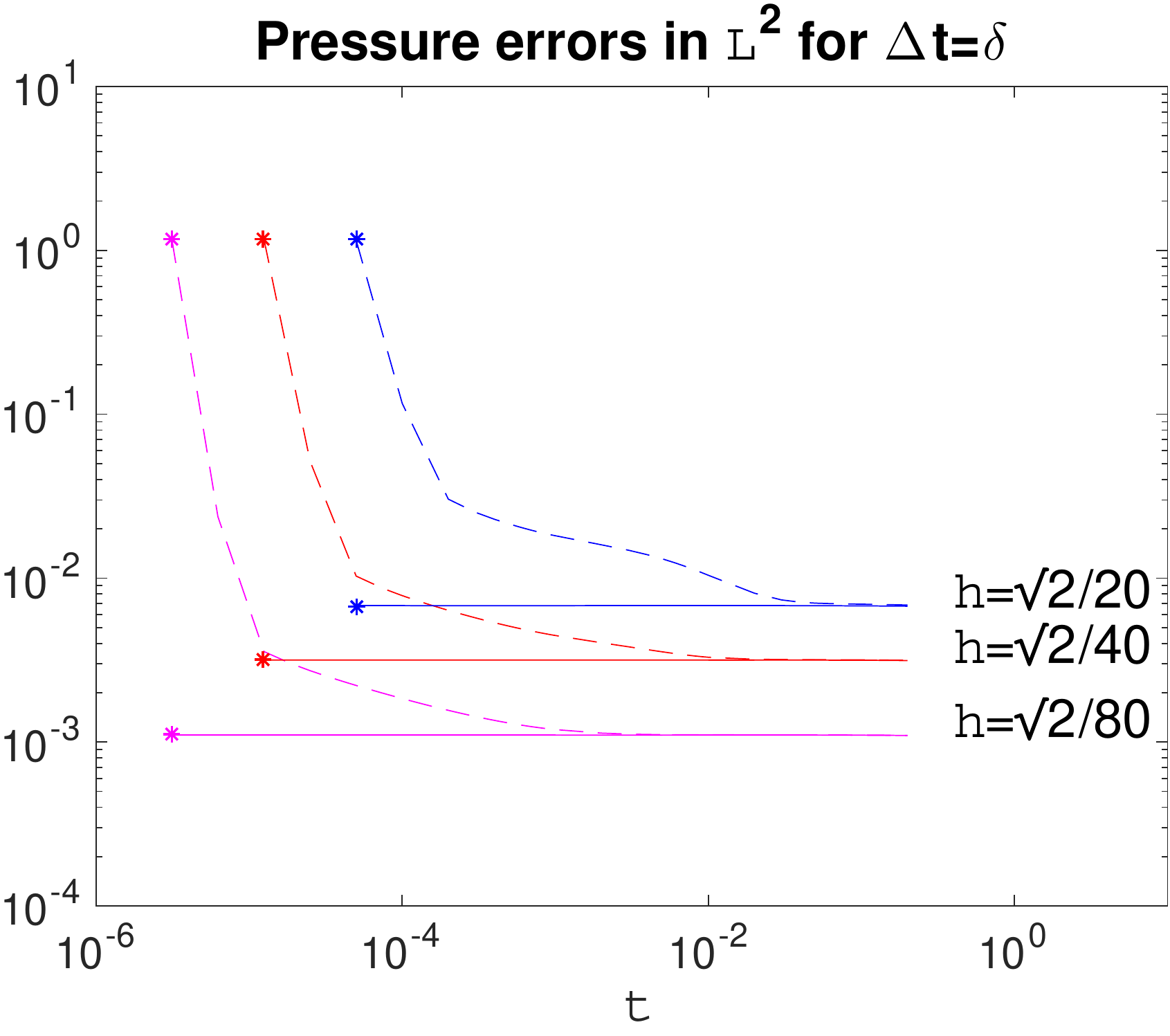}
  \end{center}
   \caption{Pressure errors $p_h^n-I_h(p(t_n))$ for $\Delta t=\delta$ y $\delta=h^2/(100\nu)$: Initial data~(\ref{eq:ini_s}) (solid line),
 and~(\ref{eq:ini_I}) (broken line).}
  \label{fig:evol1}
\end{figure}

For the evolution problem~(\ref{eq:evo_stokes}) we now study how
the choice of the initial condition affects
the errors in the method~(\ref{eq:eu_non_tilde})-(\ref{eq:eu_non_tilde2}). We
choose the forcing term~${\bg}$ so that the solution is
$$
{\bv}(x,y,t)={\mbf s}(x,y)\cos(t),\qquad q(x,y,t)=z(x,y)\cos(t),
$$
where~${\mbf s}$ and~$z$ are those in~(\ref{eq:motiv_v})-(\ref{eq:motiv_p}). We show the errors corresponding to two
different initial conditions, the first one being that given by the
linear interpolant of the true solution,
\begin{equation}
\label{eq:ini_I}
\tilde \bv_h^0=I_h(\bv(0)),\qquad q_h^0=I_h(q(0)),
\end{equation}
and the second one that given by the stabilized
 Stokes
approximation~(\ref{eq:pro_stokes})-(\ref{eq:pro_stokes2}) to~(\ref{eq:stokes})
\begin{equation}
\label{eq:ini_s}
\tilde \bv_h^0={\mbf s}_h(0)\qquad q_h^0=z_h(0)
\end{equation}
where $\hat\bg$ is chosen so that the solution is
$\bv(0)$ and~$q(0)$. According to Remark~\ref{re:remark6},
any initial data other than~(\ref{eq:ini_s}) should give an~$\mathcal{O}(1)$
error in the pressure in the first step. This can be seen
in~Fig.~\ref{fig:evol1}, where  we show
the time evolution of the errors $q_h^n - I_h(q(t_n))$, for $\delta=h^2/(100\nu)$
and decreasing values of~$h$. It can be observed that whereas for initial data given by~(\ref{eq:ini_s}) (joined by a solid line) the errors decrease with~$h$ already from the first step, they remain~$\mathcal{O}(1)$ in the first step for initial
data~(\ref{eq:ini_I}) (joined by a broken line). Nevertheless, these $\mathcal{O}(1)$ errors decay very fast with time and, for a fixed $t>0$ they decay with $h$ as well. Eventually, for
$t$ sufficiently large, they are indistinguishable from those corresponding
to initial data given by~(\ref{eq:ini_s}).

\end{document}